\documentclass[reqno]{amsart}
\usepackage[margin=3.5cm]{geometry}
\usepackage[dvipsnames]{xcolor}

\usepackage{amsfonts, amsthm, amssymb,verbatim,amscd,amsmath, blindtext}
\usepackage{comment}

\usepackage{multirow,multicol}
\usepackage{graphicx}
\graphicspath{ {./images/} }
\usepackage{enumitem,yhmath}
\usepackage{amssymb}
\usepackage{ytableau}
\usepackage{tikz}
\usetikzlibrary{decorations.markings}
\usetikzlibrary{calc}
\usepackage{tikz-cd}
\usepackage{float, pifont}
\usepackage[utf8]{inputenc}
\usepackage{mathtools,nccmath}
\usepackage{float}
\usepackage[pagebackref,colorlinks=true,citecolor=blue,linkcolor=BrickRed]{hyperref}
\usepackage{caption} 
\newtheorem{theorem}{Theorem}[section]
\newtheorem{lemma}[theorem]{Lemma}
\newtheorem{corollary}[theorem]{Corollary}
\newtheorem{proposition}[theorem]{Proposition}

\theoremstyle{remark}
\newtheorem{remark}[theorem]{Remark}
\newtheorem{question}[theorem]{Question}

\theoremstyle{definition}
\newtheorem{definition}[theorem]{Definition}

\newtheorem{example}[theorem]{Example}

\DeclareMathOperator{\mult}{mult}
\DeclareMathOperator{\cone}{cone}
\DeclareMathOperator{\ind}{Ind}

\title{Multiplicity of negative one of independence polynomials of graphs}

\author[Bhardwaj]{Om Prakash Bhardwaj}
\address{Chennai Mathematical Institute, Siruseri, Tamil Nadu 603103. India}
\email{omprakash@cmi.ac.in}

\author[Chau]{Trung Chau}
\address{Chennai Mathematical Institute, Siruseri, Tamil Nadu 603103. India}
\email{chauchitrung1996@gmail.com}

\author[Ikram]{Sayeed Ikram}
\address{Chennai Mathematical Institute, Siruseri, Tamil Nadu 603103. India}
\email{sayeedikram.ug2024@cmi.ac.in}

\author[Lather]{Gargi Lather}
\address{Chennai Mathematical Institute, Siruseri, Tamil Nadu 603103. India}
\email{gargilather@gmail.com}

\author[Nanjangud]{Vasudeva Nanjangud}
\address{Chennai Mathematical Institute, Siruseri, Tamil Nadu 603103. India}
\email{vasudevasn.ug2024@cmi.ac.in}

\author[Venugopal]{Chitra Venugopal}
\address{Chennai Mathematical Institute, Siruseri, Tamil Nadu 603103. India}
\email{chitrav@cmi.ac.in}

\begin{document}
	
    \begin{abstract} 
        We initiate the study of the multiplicity of negative one of independence polynomials of graphs. In this article, we simply refer to this as the \emph{multiplicity} of a graph. As applications, we provide a graph-theoretic description of trees whose independence complexes are contractible, give a new sufficient condition for  independence polynomials of graphs to be log-concave, and finally, determine possible pairs $(\mult_{-1}P_G, \alpha(G))$, where $P_G$ denotes the independence polynomial of $G$, and $\alpha(G)$ the independence number. The study of the pairs $(\mult_{-1}P_G, \alpha(G))$ is equivalent to finding all pairs of the numerator degree and denominator degree of the Hilbert series of the edge ideal of $G$. We also use spectral graph theory to obtain results on the multiplicity of line graphs of forests. Finally, we give some translations and applications in combinatorial commutative algebra.
    \end{abstract}

    \maketitle

\section{Introduction}

All graphs in this article are assumed to be finite and simple. For a graph $G$, let $P_G(x)$ denote the \emph{independence polynomial} of $G$, i.e., 
\[
P_G(x)=\sum_{i=0}^{\infty} g_ix^i
\]
where $g_0=1$, and $g_i$ denotes the number of independent sets of $G$ of size $i$, for each $i>0$. Independence polynomials are ubiquitous in the literature of graph theory. We refer to  \cite{MR2186466} for a comprehensive survey. These polynomials are also known in statistical physics as partition functions of hard-core lattice gas \cite{MR2130890}. The evaluation $P_G(-1)$ is sometimes referred to as the \emph{modularity} of $G$ \cite{modularity-thesis}. The values of $P_G(-1)$, and in particular, when one has $|P_G(-1)|\leq 1$, have been studied extensively in different contexts \cite{MR2393250,MR4145813,MR3518424,MR2123823,MR2240773,MR3027601}.

The independent sets of a graph $G$ form a simplicial complex, called the \emph{independence complex} of $G$, denoted by $\ind(G)$. Modularity has subtly appeared in the study of independence complex, as we have $P_G(-1)=-\tilde{\chi}(\ind(G))$, where $\tilde{\chi}(\ind(G))$ denotes the reduced Euler characteristic of $\ind(G)$. Recall that a \emph{ternary graph} is one that does not have any induced cycle of length divisible by 3. It was a conjecture by Galai and Meshulam, now proven by  Chudnovsky, Scott, Seymour, and Spirkl \cite{MR4145813}, that 
\[
\text{$G$ is a ternary graph} \Longleftrightarrow |P_H(-1)|\leq 1 \text{ for any induced subgraph $H$ of $G$}.
\]
Kim \cite{MR4403044} provided a different characterization for ternary graphs, confirming Engstr\"om's conjecture \cite{engstrom2020topologicalkalaimeshulamconjecture}:
\begin{multline*}
    \text{$G$ is a ternary graph} \Longleftrightarrow \ind(H) \text{ is either contractible or homotopy equivalent to a sphere,} \\
    \text{for any induced subgraph $H$ of $G$}.
\end{multline*}
These two characterizations are directly related since for any graph $G$, $\ind(G)$ being contractible implies that $P_G(-1)=0$, and $\ind(G)$ being homotopy equivalent to a sphere implies that $P_G(-1)=\pm 1$. Therefore, providing that $G$ is ternary, we have an equivalent statement: $\ind(G)$ is contractible if and only if $P_G(-1)=0$. A homological/topological characterization for this was obtained recently by Faridi and Holleben \cite{faridi2025sphericalcomplexes}.

In this article, we provide a graph-theoretic description of trees $T$ that satisfy $P_T(-1)=0$, or equivalently, have a contractible independence complex. This is our first main result. 

\begin{theorem}[\protect{Theorem~\ref{thm:tree-description}}]\label{thm:A}
    Let $T$ be a tree. Then the independence complex $\ind(T)$ is contractible if and only if there exists a tree $T'$ and a family of rooted trees $\mathcal{C}$ such that $T$ is isomorphic to $\mathcal{G}(T',\mathcal{C})$ as graphs.
\end{theorem}

We refer to Section~\ref{sec:pseudo-forests} for the exact definition of the grafting operation $\mathcal{G}$. As a preview, we present an example of what such trees look like. Let $T'$ be the path on three vertices. We  replace every edge of $T'$ with a path on four vertices. We call the resulting graph the \emph{3-subdivision} of $T'$. There are four new vertices compared to $T'$, and we name them $v_1,v_2,v_3,v_4$. Next, let  $\mathcal{C}=\{(T_1,u_1),(T_2,u_2),(T_3,u_3),(T_4,u_4)\}$ be a family of rooted trees. By a \emph{rooted tree}, we mean a tree together with a fixed vertex of the tree. The \emph{grafting} of $T'$ and $\mathcal{C}$, denoted by $\mathcal{G}(T',\mathcal{C})$, is defined to be the graph obtained by identifying $u_i$ with $v_i$ for each $i=1,2,3,4$. We provide a picture below.

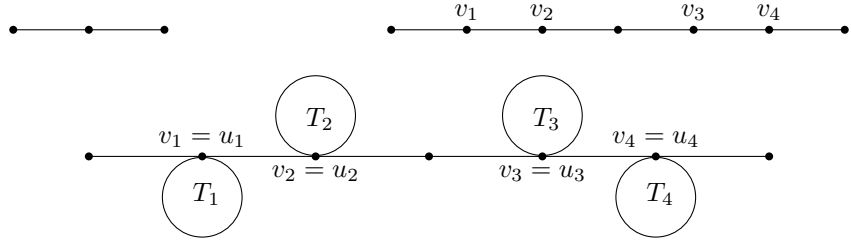
\begin{figure}[h]
    \centering
    \begin{tabular}{c}
         \begin{tikzpicture}
             \coordinate (a1) at (0,0);
			\coordinate (a2) at (1,0);
			\coordinate (a3) at (2,0);
			
			\draw (a1)--(a2)--(a3);

            \coordinate (b1) at (5,0);
			\coordinate (b2) at (6,0);
			\coordinate (b3) at (7,0);
            \coordinate (b4) at (8,0);
            \coordinate (b5) at (9,0);
            \coordinate (b6) at (10,0);
            \coordinate (b7) at (11,0);

			\draw (b1)--(b2)--(b3)--(b4)--(b5)--(b6)--(b7);

            \foreach \p in {a1,a2,a3,b1,b2,b3,b4,b5,b6,b7}
			\fill (\p) circle (1.5pt);

            \node[above] at (b2) {$v_1$};
            \node[above] at (b3) {$v_2$};
            \node[above] at (b5) {$v_3$};
            \node[above] at (b6) {$v_4$};
         \end{tikzpicture} \\\\
         \begin{tikzpicture}	
			\coordinate (b1) at (0,0);
			\coordinate (b2) at (1.5,0);
			\coordinate (b3) at (3,0);
            \coordinate (b4) at (4.5,0);
            \coordinate (b5) at (6,0);
            \coordinate (b6) at (7.5,0);
            \coordinate (b7) at (9,0);

			\draw (b1)--(b2)--(b3)--(b4)--(b5)--(b6)--(b7);

            \foreach \p in {b1,b2,b3,b4,b5,b6,b7}
			\fill (\p) circle (1.5pt);

			\node[above] at (b2) {$v_1=u_1$};
            \node[below] at (b3) {$v_2=u_2$};
            \node[below] at (b5) {$v_3=u_3$};
            \node[above] at (b6) {$v_4=u_4$};

            \draw (1.5, -0.54) circle (15pt);
            \draw (7.5, -0.54) circle (15pt);
            \draw (3, 0.54) circle (15pt);
            \draw (6, 0.54) circle (15pt);

	           \node at (1.5,-0.5) {{ $T_{1}$}};
               \node at (7.5,-0.5) {{ $T_{4}$}};
               \node at (3,0.5) {{ $T_{2}$}};
               \node at (6,0.5) {{ $T_{3}$}};
		\end{tikzpicture} 
    \end{tabular}
    \caption{A tree $T'$, its 3-subdivision, and the grafting $\mathcal{G}(T',\mathcal{C})$.}
    \label{fig:example-intro}
\end{figure}

We remark that determining $P_T(-1)$ for a tree $T$ has been studied before in different areas of mathematics \cite{faridi2025sphericalcomplexes,MR2791291,pham2026contractibleindependencecomplexestrees}. However, the known results are based on the output after inputting $T$ into an algorithm, while Theorem~\ref{thm:A} provides the explicit graph-theoretic description of such trees. In this line of attack, we also provide an algorithmic method to determine $P_T(-1)$, for a larger class of graphs: pseudo-forests (Theorem~\ref{thm:branch-algorithm}). Our result is closest to \cite{pham2026contractibleindependencecomplexestrees}, with the difference in the output: a tree reduces to a union of isolated vertices via our process, but it reduces to a path via their process.

Next, we initiate the study of the multiplicity of $x=-1$ of the independence polynomial $P_G(x)$, which we denote by $\mult_{-1} P_G$. Throughout this article, we will simply refer to this as the \emph{multiplicity of a graph}. Let $\alpha(G)$ denote the independence number of a graph $G$, i.e., $\alpha(G)=\deg P_G$.   Recall that a polynomial
\[
a_0+a_1x+\cdots +a_nx^n
\]
with positive coefficients is called \emph{log-concave} if $a_{i}^2\geq a_{i-1}a_{i+1}$ for any $1\leq i \leq n-1$. A log-concave polynomial is always unimodal. It is a problem of great interest in graph theory to determine which graphs have log-concave/unimodal independence polynomials \cite{Chudnovsky, Hamidoune, Horrocks, Levit-Carpathian, Levit-Congressus, Schwenk, Stanley}.  Our next main result gives a new such class of graphs.

\begin{theorem}[\protect{Theorem~\ref{thm:log-concave}}]
    If $\mult_{-1}(G)\geq \alpha(G)-2$, then $P_G$ is log-concave.
\end{theorem}

Determining all pairs of $(\mult _{-1} P_G, \alpha(G))$ is the same as determining all pairs of dimensions and degrees of the $h$-polynomial of edge ideals, which are all pairs of degrees of the numerator and denominator of the corresponding Hilbert series written as a rational function (see Section~\ref{sec:algebra} for definitions). To that end, our final goal is to determine the following sets:
\begin{align*}
    \mathcal{MI}(n)&\coloneqq \{ (\mult_{-1} P_G, \alpha(G)) \mid \text{$G$ is a graph on $n$ vertices} \}, \\
    \mathcal{MI}^c(n)&\coloneqq \{ (\mult_{-1} P_G, \alpha(G)) \mid \text{$G$ is a connected graph on $n$ vertices} \},   \\
    \mathcal{MI}&\coloneqq \bigcup_{n=1}^{\infty} \mathcal{MI}(n), \quad \text{and} \quad \mathcal{MI}^c\coloneqq\bigcup_{n=1}^{\infty} \mathcal{MI}^c(n),
\end{align*}
for each $n\geq 1$. Except for $\mathcal{MI}^c(n)$, we have a complete description.

\begin{theorem}[\protect{Theorems~\ref{cor:MI-all-connected}} and \ref{thm:MI-n-all-graphs}]\label{thm:C}
    We have
    \begin{align*}
        \mathcal{MI}(n) &= \{(a,b)\in \mathbb{Z}_{\geq 0}^2\mid 0\leq a< b \leq n-1 \} \cup \{(n,n)\} \text{ for each $n\geq 1$},\\
        \mathcal{MI} &= \{(a,b)\in \mathbb{Z}_{\geq 0}^2\mid 0\leq a\leq b \} \setminus  \{(0,0)\},\\
        \mathcal{MI}^c &= \{(a,b)\in \mathbb{Z}_{\geq 0}^2\mid 0\leq a< b  \} \cup \{(1,1)\}.
    \end{align*}
\end{theorem}

The set $\mathcal{MI}^c(n)$ is harder to determine, as the requirement that the graph be connected imposes additional restrictions on its invariants. We obtain bounds for $\mathcal{MI}^c(n)$ together with many realizable points on the boundary. The following is proved throughout Section~\ref{sec:MI}.

\begin{theorem}[\protect{Lemma~\ref{lem:MI-n-connected-upperbound} and Theorem~\ref{thm:lower-bound-MI-c}}]\label{thm:intro-bounds}
    Let $n\geq 2$. We have
    \[
    \mathcal{MI}^c(n)\subseteq \{(a,b)\in \mathbb{Z}_{\geq 0}^2 \mid 0\leq a <b \leq n-2\} \cup\{(0,n-1)\}.
    \]
    Moreover, if $n$ is odd, then
    \begin{equation*}
        \{(a,b)\in \mathbb{Z}_{\geq 0}^2\mid 0\leq a<b\leq n-2 \text{ and } a\leq \lceil n/2 \rceil-1 \} \setminus \{(\lceil n/2  \rceil-1,n-2)\}\cup \{(0,n-1)\}\subseteq \mathcal{MI}^c(n),
    \end{equation*}
    and if $n$ is even, then 
    \begin{equation*}
        \{(a,b)\in \mathbb{Z}_{\geq 0}^2\mid 0\leq a<b\leq n-2 \text{ and } a\leq n/2 -1 \} \cup \{(0,n-1)\}\subseteq \mathcal{MI}^c(n).
    \end{equation*}
\end{theorem}

In fact, the lower bound for $\mathcal{MI}^c(n)$ in Theorem~\ref{thm:intro-bounds} is exactly the set $\mathcal{MI}^c(n)$ itself for $2\leq n\leq 8$. Whether this holds for larger $n$ reduces to Question~\ref{ques} and remains open.

Spectral theory is the study of eigenvalues of a matrix, i.e., the roots of the characteristic polynomial of a matrix. Spectral graph theory studies graph properties and invariants via the use of spectral theory. It is thus no surprise that techniques in spectral graph theory can produce interesting information on the multiplicity of graphs. In Section~\ref{sec:spectral}, we translate some results in spectral graph theory into our context, giving multiplicity of the line graph of forests.

Finally, we remark that $\mult_{-1}P_G$ is exactly negative one multiplied with the $\mathfrak{a}$-invariant of the edge ideal of $G$ \cite{biermann-et-al.2026}. In other words, every result we obtain in this article has an equivalent statement in combinatorial commutative algebra. We provide the translations for some, together with some applications, in Section~\ref{sec:algebra}. In particular, we obtain many results regarding possible pairs of $\mathfrak{a}$-invariants and dimension of edge ideals of graphs. These results align with the large (and growing) literature on constructing graphs with given parameters \cite{MR3805868,MR3758241,MR4508184,MR3518424,erdos1960graphs}, and of finding tuples of invariants of edge ideals \cite{MR4494587,MR4302169,MR4270555,MR3919621,MR4585898}.

\section*{Acknowledgements}

We thank Priyavrat Deshpande, Takayuki Hibi, Do Trong Hoang, Thiago Holleben, and Adam van Tuyl for helpful feedback on an earlier version of this article. The first named author is supported by ANRF National Postdoctoral Fellowship. The first, second, fourth, and sixth named authors are supported by the Infosys Foundation.

\section{Graph theory}

Let $G=(V(G),E(G))$ be a finite simple graph. For a vertex $v\in V(G)$, a vertex that forms an edge with $v$ in $G$ is called its \emph{neighbor} in $G$. The set of all neighbors of $v$ is denoted by $N_G(v)$, and $N_G[v]\coloneqq N_G(v)\cup \{v\}$ is called the \emph{closed neighborhood} of $v$. A vertex $v\in V(G)$ is called a \emph{pendant vertex} of $G$ if $|N_G(v)|=1$. The unique neighbor of a pendant vertex is called a \emph{support vertex}. Equivalently, a support vertex is a vertex whose neighborhood contains a pendant vertex. When $G$ is well understood from the context, we will drop the subscripts. For a set of vertices $U\subseteq V(G)$, let $G\setminus U$ denote the induced subgraph of $G$ with the vertex set $V(G)\setminus U$.

We recall the following classical result on computing independence polynomials.

\begin{lemma}[\protect{\cite{MR1263751}}]\label{lem:basic-identities}
    Let $G$ be a finite simple graph and $u\in V(G)$. Then \[
    P_G(x) = P_{G\setminus \{u\}}(x)+xP_{G\setminus N[u]}(x).\]
\end{lemma}

We recall a common construction: cone over a set. Let $G$ be a finite simple graph on $[n]$ and $U\subseteq V(G)$ be a set of vertices of $G$. The \emph{cone of $G$ over $V(G)\setminus U$}, denoted by $\cone(G,U)$, is the graph on $[n+1]$ with
\[
E(\cone(G,U)) = E(G)\cup \{ \{i,n+1\} \mid i\notin U  \}.
\]
It is straightforward to obtain the independence polynomial of $\cone(G,U)$ when $U$ is an independent set.

\begin{lemma}\label{lem:cone-ind-poly}
    Let $G$ be a finite simple graph and $U$ an independent set of $G$. Then \[
    P_{\cone(G,U)}(x)=P_G(x)+x(1+x)^{|U|}.
    \]
    Consequently, we have the following:
    \begin{enumerate}
        \item if $\mult_{-1}P_G<|U|<\alpha(G)$, then
        \[
        \mult_{-1}P_{\cone(G,U)} = \mult_{-1} P_G \quad \text{and}\quad \alpha(\cone(G,U))=\alpha(G);
        \]
        \item if $|U|<\mult_{-1} P_G$, then
        \[
        \mult_{-1}P_{\cone(G,U)} = |U| \quad \text{and}\quad \alpha(\cone(G,U))=\alpha(G);
        \]
        \item if $|U|=\alpha(G)$ and $G$ has an edge, then
        \[
        \mult_{-1}P_{\cone(G,U)} = \mult_{-1} P_G \quad \text{and}\quad \alpha(\cone(G,U))=\alpha(G)+1.
        \]
    \end{enumerate}
\end{lemma}

\begin{proof}
    Assume that $V(\cone(G,U))= V(G)\cup \{u\}$. Then the independence polynomial formula follows from Lemma~\ref{lem:basic-identities}, remarking that $\cone(G,U)\setminus N_{\cone(G,U)}[u]=U$ is an independent set of size $|U|$. The second statement then follows straightforwardly.
\end{proof}

For two graphs $G$ and $H$, their \emph{union}, denoted by $G\sqcup H$, is the graph obtained by merging the vertex and edge sets, i.e., 
\[
V(G\sqcup H) = V(G)\sqcup V(H) \quad\text{and} \quad E(G\sqcup H) = E(G)\sqcup E(H).
\]
Here we use disjoint unions in the definition to emphasize that $V(G)$ and $V(H)$ are considered disjoint in this construction. The following result is standard, hence we do not provide a~proof.

\begin{lemma}\label{lem:union-graphs}
    Let $G$ and $H$ be two finite simple graphs. Then $P_{G\sqcup H}(x)=P_G(x)P_H(x)$. In particular, we have
    \[
    \mult_{-1} P_{G\sqcup H}= \mult_{-1}P_G + \mult_{-1} P_H \quad \text{and}\quad \alpha(G\sqcup H)=\alpha(G)+\alpha(H).
    \]
\end{lemma}

We recall some common graphs. Let $n\geq 1$ be an integer. The \emph{star graph} and \emph{complete graph} on $[n]$, denoted by $S_n$ and $K_n$, respectively, are the graphs with the edge sets:
\[
E(S_n)=\{\{i,n\}\colon 1\leq i\leq n-1\} \quad \text{and}\quad E(K_n)= \{\{i,j\}\colon 1\leq i<j\leq n \}.
\]
These two graphs, together with the graph of $n$ isolated vertices, can be fully characterized using independence polynomials. The next three results are standard, and we leave them as exercises to interested readers.

\begin{lemma}\label{lem:n-isolated-vertices}
    Let $G$ be a graph on $n\geq 1$ vertices. The following are equivalent:
    \begin{enumerate}
        \item $G$ is $\sqcup_{k=1}^n K_1$, the graph of $n$ isolated vertices;
        \item $\alpha(G)=n$;
        \item $\mult_{-1} P_G=\alpha(G)$.
    \end{enumerate}
\end{lemma}




\begin{lemma}\label{lem:complete-graph}
    Let $G$ be a graph on $n\geq 1$ vertices. The following are equivalent:
    \begin{enumerate}
        \item $G$ is $K_n$, the complete graph on $n$ vertices;
        \item $\alpha(G)=1$;
        \item $P_G(x)=1+nx$.
    \end{enumerate}
\end{lemma}

\begin{lemma}\label{lem:star-graph}
    Let $G$ be a connected graph on $n\geq 2$ vertices. The following are equivalent:
    \begin{enumerate}
        \item $G$ is $S_n$, the star graph on $n$ vertices;
        \item $\alpha(G)=n-1$;
        \item $P_{G}(x)=(1+x)^{n-1}+x$.
    \end{enumerate}
\end{lemma}

\section{Pseudo-forests with positive multiplicity}\label{sec:pseudo-forests}

The goal of this section is twofold. The first is to establish an algorithmic method to determine the value of $P_G(-1)$, where $G$ is a pseudo-forest.  The second objective is to provide an explicit description of trees $T$ with $P_T(-1)=0$. We start with a key lemma.

\begin{lemma}[\protect{\cite[Lemma~2.1]{MR3027601}}] \label{LevitMandres09}
    Let $G$ be a finite simple graph and $v$ be a support vertex of $G$.~Then
    \[
    P_G(-1)=(-1)P_{G\setminus N[v]}(-1).
    \]
\end{lemma}

\begin{definition}
    For a finite simple graph $G$, a \emph{support sequence} of $G$ is a sequence of vertices $v_1,\dots, v_t$ such that $v_i$ is a support vertex of $G\setminus \cup_{j=1}^{i-1} N[v_j]$ for any $i\in [t]$.
\end{definition}

By definition, a support sequence $v_1,\dots, v_t$ of $G$ is maximal if and only if $G\setminus \cup_{i=1}^{t} N_{G}[v_i]$ has no support vertex if and only if $G\setminus \cup_{i=1}^{t} N_{G}[v_i]$ has no pendant vertex. It is noteworthy that two maximal support sequences may have different lengths, and the resulting graph $G\setminus \cup_{i=1}^{t} N_{G}[v_i]$ is dependent on the sequence.

\begin{example}
    Let $G$ be the following tree.
    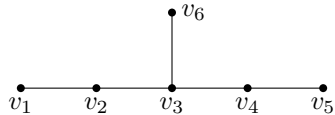
\begin{figure}[h]
        \centering
        \begin{tikzpicture}
             \coordinate (v1) at (0,0);
		\coordinate (v2) at (1,0);
		\coordinate (v3) at (2,0);
		\coordinate (v4) at (3,0);
		\coordinate (v5) at (4,0);
		\coordinate (v6) at (2,1);
        
\node[below] at (v1)  {$v_1$};
\node[below] at (v2)  {$v_2$};
\node[below] at (v3) {$v_3$};
\node[below] at (v4) {$v_4$};
\node[below] at (v5) {$v_5$};
\node[right] at (v6) {$v_6$};

\fill (v1) circle (1.5pt);
		\fill (v2) circle (1.5pt);
		\fill (v3) circle (1.5pt);
		\fill (v4) circle (1.5pt);
		\fill (v5) circle (1.5pt);
		\fill (v6) circle (1.5pt);

        \draw (v1)--(v2)--(v3)--(v4)--(v5);
        \draw (v3)--(v6);
\end{tikzpicture}
        \caption{A graph $G$ with two maximal support sequences of different lengths.}
        \label{fig:example-branch-sequence}
    \end{figure}\\
    It is straightforward that $v_3$ and $v_2,v_4$ are both maximal support sequences of $G$. Moreover, we have $G\setminus N_G[v_3]=K_1\sqcup K_1$ and $G\setminus (N_G[v_2]\cup N_G[v_4])=K_1$.
\end{example}

Recall that a \emph{pseudo-forest} is a graph $G$ such that any connected component of $G$ has at most one cycle. We shall compute the value of $P_G(-1)$ for any pseudo-forest $G$.

\begin{theorem}\label{thm:branch-algorithm}
    Let $G$ be a pseudo-forest and $v_1,\dots, v_t$ a maximal support sequence of $G$ for some integer $t$. Then $G \setminus \cup_{i=1}^t N_G[v_i]$ is a union of cycles and isolated vertices, and 
    \[
    P_G(-1)=(-1)^t P_{G \setminus \cup_{i=1}^t N_G[v_i]} (-1).
    \]
    Moreover, set $H=G \setminus \cup_{i=1}^t N_G[v_i]$, and
    \begin{align*}
        a&\coloneqq \# \text{$k$-cycles of $H$ where $k=0$ (mod $6$)},\\
        b&\coloneqq \# \text{$k$-cycles of $H$ where $k=2,4$ (mod $6$)},\\
        c&\coloneqq \# \text{$k$-cycles of $H$ where $k=3$ (mod $6$)},\\
        d&\coloneqq \# \text{isolated vertices of $H$}.
    \end{align*}
    Then 
    \[
    P_G(-1)= \begin{cases}
        0 & \text{if } d>0,\\
        (-1)^{t+b+c} (2)^{a+c}  & \text{if } d=0.
    \end{cases}
    \]
\end{theorem}

\begin{proof}
    Since $G$ is a pseudo-forest, the only subgraphs of $G$ that do not have any pendant vertex are unions of cycles and isolated vertices. The first statement then follows from the fact that $G \setminus \cup_{i=1}^t N_G[v_i]$ has no pendant vertex, and Lemma~\ref{LevitMandres09}. The second statement comes from Lemma~\ref{lem:union-graphs}, and the formula 
    \[
    P_{C_n}(-1)= \begin{cases}
        2 & \text{if } n=0 \text{ (mod 6)}\\
        1 & \text{if } n=1,5 \text{ (mod 6)}\\
        -1 & \text{if } n=2,4 \text{ (mod 6)}\\
        -2 & \text{if } n=3 \text{ (mod 6)}
    \end{cases} \tag*{(\cite[Lemma 3.2.2]{modularity-thesis}).\qedhere}
    \]
\end{proof}

We obtain a quick corollary about the independence complex of $G$ where $G$ is a special class of pseudo-forests.

\begin{corollary}\label{cor:ternary-pseudo-forest}
    Let $G$ be a ternary pseudo-forest, i.e., a pseudo-forest that does not have any cycle of length divisible by 3. Then the independence complex of $G$ is contractible if and only if for any maximal support sequence $v_1,\dots, v_t$ of $G$, the graph $G \setminus \cup_{i=1}^t N_G[v_i]$ has an isolated~vertex.
\end{corollary}

Next, we characterize trees whose independence polynomial vanishes at $-1$. For this, we introduce some terminology.

\begin{definition}
Let $T$ be a tree. The \emph{$3$-subdivision} of $T$, denoted by $T^*$, is the tree obtained from $T$ by replacing each edge $e=\{u,v\}$ with a path of length $3$. More precisely, for each edge $e=\{u,v\}$, introduce two new vertices $w_{u,e}$ and $w_{v,e}$, and replace $e$ with the edges $\{u,w_{u,e}\}$, $\{w_{u,e},w_{v,e}\}$, and $\{w_{v,e},v\}$.
\end{definition}

We present an example in Figure~\ref{fig:T*} to illuminate the concept.

\begin{figure}[h]
    \centering
    \begin{tabular}{cccc}
        \begin{tikzpicture}[scale=1.3]
		\coordinate (v2) at (0,0);
		\coordinate (v1) at (0,1);
		\coordinate (v3) at (-1,-1);
		\coordinate (v4) at (1,-1);
		\coordinate (v5) at (0,-2);
		\coordinate (v6) at (2,-2);
		
		\draw (v2) -- (v1);
		\draw (v2) -- (v3);
		\draw (v2) -- (v4);
		\draw (v4) -- (v5);
		\draw (v4) -- (v6);

		\fill (v1) circle (1.5pt);
		\fill (v2) circle (1.5pt);
		\fill (v3) circle (1.5pt);
		\fill (v4) circle (1.5pt);
		\fill (v5) circle (1.5pt);
		\fill (v6) circle (1.5pt);

		\node[above] at (v1) {1};
		\node[right] at (v2) {2};
		\node[left] at (v3) {3};
		\node[above right] at (v4) {4};
		\node[below] at (v5) {5};
		\node[right] at (v6) {6};
	\end{tikzpicture} &&&  \begin{tikzpicture}	
			\coordinate (a1) at (0,1.5);
			\coordinate (a2) at (0,1);
			\coordinate (a3) at (0,.5);
			\coordinate (v2) at (0,0);
			
			\draw (a1)--(a2)--(a3)--(v2);

			\coordinate (b1) at (-0.5,-0.4);
			\coordinate (b2) at (-1.0,-0.8);
			\coordinate (v3) at (-1.5,-1.2);
			
			\draw (v2)--(b1)--(b2)--(v3);

			\coordinate (c1) at (0.5,-0.4);
			\coordinate (c2) at (1,-0.8);
			\coordinate (v4) at (1.5,-1.2);
			
			\draw (v2)--(c1)--(c2)--(v4);

			\coordinate (d1) at (2,-1.6);
			\coordinate (d2) at (2.5,-2.0);
			\coordinate (v6) at (3,-2.4);
			
			\draw (v4)--(d1)--(d2)--(v6);

			\coordinate (e1) at (1,-1.6);
			\coordinate (e2) at (.5,-2.0);
			\coordinate (v5) at (0,-2.4);
			
			\draw (v4)--(e1)--(e2)--(v5);

			\foreach \p in {a1,a2,a3,v2,b1,b2,v3,c1,c2,v4,d1,d2,v6,e1,e2,v5}
			\fill (\p) circle (1.5pt);

			\node[above] at (a1) {1};
			\node[left] at (v2) {2};
			\node[left] at (v3) {3};
			\node[right] at (v4) {4};
			\node[below] at (v5) {5};
			\node[right] at (v6) {6};

			\node[right] at (a2) {{\tiny $w_{1,(1,2)}$}};
			
			\node[right] at (a3) {{\tiny $w_{2,(1,2)}$}};
			
			\node[right] at (c1) {{\tiny $w_{2,(2,4)}$}};
			
			\node[right] at (c2) {{\tiny $w_{4,(2,4)}$}};
			
			\node[left] at (b1) {{\tiny $w_{2,(2,3)}$}};
			
			\node[left] at (b2) {{\tiny $w_{3,(2,3)}$}};
			
			\node[right] at (d1) {{\tiny $w_{4,(4,6)}$}};
			
			\node[right] at (d2) {{\tiny $w_{6,(4,6)}$}};
			
			\node[left] at (e1) {{\tiny $w_{4,(4,5)}$}};
			
			\node[left] at (e2) {{\tiny $w_{5,(4,5)}$}};
		\end{tikzpicture}
    \end{tabular}
    \caption{A tree $T$ and its 3-subdivision $T^*$.}
    \label{fig:T*}
\end{figure}
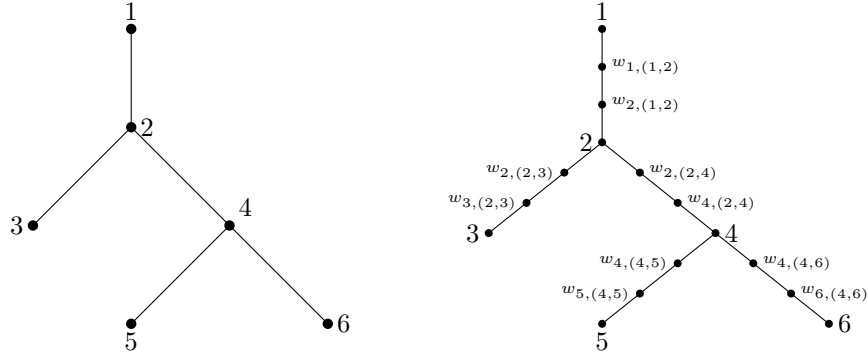

Let $T=(V,E)$ be a tree. For each edge $e \in E$ and each vertex $u \in e$, let
$$
\mathcal{C}=\{(T_{u,e}, r_{u,e}) : e \in E,\; u \in e\}
$$
denote an indexed family of rooted graphs, where $T_{u,e}=(V_{u,e},E_{u,e})$ is a graph and $r_{u,e} \in V_{u,e}$ is a fixed vertex in $V_{u,e}$. We call such a $\mathcal{C}$ a \emph{family of rooted graphs corresponding to $T$}. Here by a \emph{rooted graph}, we mean a pair of a graph and a vertex of its.

\begin{definition}
For a tree $T=(V,E)$ and a family of rooted graphs $\mathcal{C}$ corresponding to $T$,  the \emph{grafting} of $\mathcal{C}$ on $T$, denoted by $\mathcal{G}(T,\mathcal{C})$, is defined as the graph obtained from $T^*$ by identifying the root $r_{u,e}$ of $T_{u,e}$ to the vertex $w_{u,e}$ of $T^*$ for every pair $(u,e)$ with $u \in e$ and $e \in E$.

Formally,
$$
\mathcal{G}(T,\mathcal{C}) \;=\;
\left( \, T^* \;\sqcup\;\bigsqcup_{\substack{e \in E \\ u \in e}} T_{u,e} \,\right)\Big/ \sim,
$$
where the relation $\sim$ identifies the root $r_{u,e}$ of each $T_{u,e}$ with the vertex $w_{u,e}$ in $T^*$. 
\end{definition}

\begin{remark}
    As can be seen later (the proof of Theorem~\ref{thm:tree-description}), when we study $\mathcal{G}(T,\mathcal{C})$, the structure of the graphs in $\mathcal{C}$  plays a minimal role. When there are no specific restriction on $\mathcal{C}$, we will refer to $\mathcal{G}(T,\mathcal{C})$ as simply a grafting on $T$.
\end{remark}

\begin{example}
Let $T$ be the tree as in Figure~\ref{fig:T*}. The edge set $E= \{(1,2),(2,3),(2,4),(4,5),(4,6)\}$. Consider the following family of rooted trees
\[ \mathcal{C} = \{(T_{1,(1,2)},r_{1,(1,2)}), (T_{2,(1,2)},r_{2,(1,2)}), \ldots, (T_{4,(4,6)},r_{4,(4,6)}), (T_{6,(4,6)},r_{6,(4,6)})\}.\]
The grafting $\mathcal{G}(T,\mathcal{C})$ of $\mathcal{C}$ on $T$ will look like the following tree

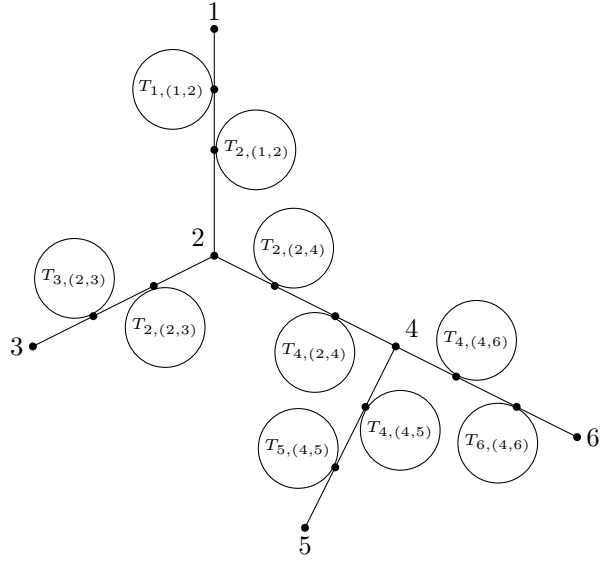
\begin{center}
\begin{tikzpicture}
	\coordinate (a1) at (0,3);
	\coordinate (a2) at (0,2.2);
	\coordinate (a3) at (0,1.4);
	\coordinate (v2) at (0,0);
	
	\draw (a1)--(a2)--(a3)--(v2);

	\coordinate (b1) at (-0.8,-0.4);
	\coordinate (b2) at (-1.6,-0.8);
	\coordinate (v3) at (-2.4,-1.2);
	
	\draw (v2)--(b1)--(b2)--(v3);

	\coordinate (c1) at (0.8,-0.4);
	\coordinate (c2) at (1.6,-0.8);
	\coordinate (v4) at (2.4,-1.2);
	
	\draw (v2)--(c1)--(c2)--(v4);
	
	\coordinate (d1) at (3.2,-1.6);
	\coordinate (d2) at (4.0,-2.0);
	\coordinate (v6) at (4.8,-2.4);
	
	\draw (v4)--(d1)--(d2)--(v6);

	\coordinate (e1) at (2.0,-2.0);
	\coordinate (e2) at (1.6,-2.8);
	\coordinate (v5) at (1.2,-3.6);
	
	\draw (v4)--(e1)--(e2)--(v5);

	\foreach \p in {a1,a2,a3,v2,b1,b2,v3,c1,c2,v4,d1,d2,v6,e1,e2,v5}
	\fill (\p) circle (1.5pt);

	\node[above] at (a1) {1};
	\node[above left] at (v2) {2};
	\node[left] at (v3) {3};
	\node[above right] at (v4) {4};
	\node[below] at (v5) {5};
	\node[right] at (v6) {6};
	
	\draw (-0.55,2.2) circle (15pt);
	\draw (0.55,1.4) circle (15pt);
	
	\draw (-0.65,-.95) circle (15pt);
	\draw (-1.85,-.3) circle (15pt);

	\draw (1.05, 0.10) circle (15pt);
	\draw (1.33, -1.28) circle (15pt);
			
	\draw (2.46, -2.31) circle (15pt);
	\draw (1.11, -2.55) circle (15pt);	
		
	\draw (3.47, -1.12) circle (15pt);
	\draw (3.75, -2.48) circle (15pt);

	\node[left] at (a2) {{\tiny $T_{1,(1,2)}$}};	
	\node[right] at (a3) {\tiny $T_{2,(1,2)}$};
	
	\node at (1.05, 0.10) {\tiny $T_{2,(2,4)}$};
	\node at (1.33, -1.28) {\tiny $T_{4,(2,4)}$};
	
	\node at (-0.65,-.95) {\tiny $T_{2,(2,3)}$};
	\node at (-1.85,-.3) {\tiny $T_{3,(2,3)}$};
	
	\node at (2.46, -2.31) {\tiny $T_{4,(4,5)}$};
	\node at (1.11, -2.55) {\tiny $T_{5,(4,5)}$};
	
	\node at (3.47, -1.12) {\tiny $T_{4,(4,6)}$};
	\node at (3.75, -2.48) {\tiny $T_{6,(4,6)}$};
\end{tikzpicture}
\end{center}
\end{example}

\begin{theorem}\label{thm:tree-description}
A tree $T$ satisfies $P_T(-1)=0$ if and only if there exist a tree $T'$ and a family of rooted trees $\mathcal{C}$ such that $
T \cong \mathcal{G}(T',\mathcal{C})$.
\end{theorem}

\begin{proof}
We prove the forward implication by induction on the number of vertices in $T$.

If $|V(T)|=1$, then $T$ consists of a single vertex. In this case $P_T(-1)=0$, and we may take $T'=T$ and $\mathcal{C}=\emptyset$. Hence the statement holds.

Assume the statement holds for all trees with fewer than $n$ vertices, and let $T$ be a tree with $|V(T)|=n$ such that $P_T(-1)=0$. Since $T$ is a tree, it has a pendant vertex. Let $u$ be a pendant vertex and let $v$ be its unique neighbor. By Lemma \ref{LevitMandres09},
$P_T(-1)=(-1)\cdot P_{T \setminus N[v]}(-1).$ Since $P_T(-1)=0$, it follows that $P_{T \setminus N[v]}(-1)=0.$
Clearly $T \setminus N[v]$ is a forest. Let its connected components be $T_1,\dots,T_k$. Then
$P_{T \setminus N[v]}(-1)=\prod_{i=1}^k P_{T_i}(-1)=0,$
so there exists at least one component, say $T_i$, such that $P_{T_i}(-1)=0$. By the induction hypothesis, there exist a tree $T'_i$ and a family $\mathcal{C}_i$ such that
$$T_i \cong \mathcal{G}(T'_i,\mathcal{C}_i).$$

Since $T_i$ arises from removing $N[v]$, there exist vertices $u_i \in V(T_i)$ and $v_i \in N(v)$ such that $\{u_i,v_i\}\in E(T)$.

We now reconstruct $T$ from $T_i$, for which we consider the following two cases. 

\noindent \textbf{Case 1:}
Suppose that $u_i \in V(T'_i)$. We define a tree $T'$ and an indexed family $\mathcal{C}$ as follows. Let $$V(T') = V(T'_i)\cup\{u\}, \qquad E(T') = E(T'_i)\cup\{\{u,u_i\}\}.$$
For the pair $(u,\{u,u_i\})$, set $w_{u,\{u,u_i\}}=r_{u,\{u,u_i\}}=v$, and let $T_{u,\{u,u_i\}}$ be the subtree of $T$ rooted at $v$ obtained after deleting $u$ and $v_i$. Similarly, for the pair $(u_i,\{u,u_i\})$, set $w_{u_i,\{u,u_i\}}=r_{u_i,\{u,u_i\}}=v_i$, and let $T_{u_i,\{u,u_i\}}$ be the subtree of $T$ rooted at $v_i$ obtained after deleting $v$ and $u_i$. Finally, define $$\mathcal{C}=\mathcal{C}_i \cup \{(T_{u,\{u,u_i\}},r_{u,\{u,u_i\}}), (T_{u_i,\{u,u_i\}},r_{u_i,\{u,u_i\}})\}.$$
Then, $T \cong \mathcal{G}(T',\mathcal{C})$, as required.

\noindent \textbf{Case 2:} Suppose, $u_i$ lies in one of the trees $T_{u',e'}$ of $\mathcal{C}_i$, say $u_i \in V(T_{u',e'})$ for some $(T_{u',e'}, r_{u',e'}) \in \mathcal{C}_i$. We define $T'$ and $\mathcal{C}$ as follows.

Set $T' = T'_i$. Let $T''$ be the subtree of $T$ rooted at $u_i$ obtained by deleting $T_i \setminus \{u_i\}$ from $T$, and define
$$ 
T'_{u',e'} = T_{u',e'} \cup T''.
$$
Now define
$$ 
\mathcal{C} = \big(\mathcal{C}_i \setminus \{(T_{u',e'}, r_{u',e'})\}\big) \cup \{(T'_{u',e'}, r_{u',e'})\},
$$ 
where $(T'_{u',e'}, r_{u',e'})$ is indexed by the same pair $(u',e')$ as $(T_{u',e'}, r_{u',e'})$.

In this case as well, we obtain $T \cong \mathcal{G}(T',\mathcal{C}),$ as desired

 For the converse, we need to prove that for any tree $T$ and any indexed family $\mathcal{C}$ of rooted trees indexed by pairs of vertices and edges $(u,e)$ of $T$, where the vertex $u$ is an adjacent vertex of edge $e$, $P_{\mathcal{G}(T,\mathcal{C})}(-1) = 0.$ Let $u$ be a pendant vertex of $\mathcal{G}(T,\mathcal{C})$ which is also a pendant vertex of $T$, and $v$ be the unique neighborhood of $u$. Thus, by Lemma~\ref{LevitMandres09}, we have 
\[ P_{\mathcal{G}(T,\mathcal{C})}(-1) = (-1)P_{\mathcal{G}(T,\mathcal{C}) \setminus N[v]}(-1).\]
Now, observe that $\mathcal{G}(T,\mathcal{C}) \setminus N[v]$ is the disjoint union of $\mathcal{G}(T\setminus \{v\},\mathcal{C}\setminus \{T_{u,e},T_{v,e}\})$, $T_{u,e}\setminus N[r_{u,e}]$, and $T_{v,e}\setminus \{r_{v,e}\}.$ Therefore, by Lemma~\ref{lem:union-graphs}, we get
 \[P_{\mathcal{G}(T,\mathcal{C})}(-1) = (-1)P_{\mathcal{G}(T\setminus \{v\},\mathcal{C}\setminus \{T_{u,e},T_{v,e}\})}(-1)P_{T_{u,e}\setminus N[r_{u,e}]}(-1)P_{T_{v,e}\setminus \{r_{v,e}\}}(-1).\]
 Now, consider the graph $\mathcal{G}(T\setminus \{v\},\mathcal{C}\setminus \{T_{u,e},T_{v,e}\})$ and repeat the same process by choosing a pendant vertex. Continuing this process and using Lemmas~\ref{LevitMandres09},~\ref{lem:union-graphs} iteratively, we get
 \[P_{\mathcal{G}(T,\mathcal{C})}(-1) = (-1)^{\vert V(T) - 1\vert}P_{K_1}(-1) \prod_{\substack{e \in E \\ u,v \in e}}P_{T_{u,e}\setminus N[r_{u,e}]}(-1)P_{T_{v,e}\setminus \{r_{v,e}\}}(-1).\]
 Since $P_{K_1}(-1) = 0$, we get $P_{\mathcal{G}(T,\mathcal{C})}(-1) = 0$. This completes the proof.
\end{proof}

\begin{remark}
    A closer look at the proof of Theorem~\ref{thm:tree-description} reveals that in fact if $T$ is a tree and $\mathcal{C}$ a family of rooted graphs corresponding to $T$, then $P_{\mathcal{G}(T,\mathcal{C})} (-1)=0$. Thus the grafting operation gives more graphs $G$ such that $P_G(-1)=0$. It is straightforward to see that if $G$ is a connected graph with at least one edge, then the whiskered graph $\mathcal{W}(G)$ (see Definition~\ref{def:whiskered}) is a grafting on the path graph on two vertices, $K_2$. On the other side of the spectrum, not all graphs $G$ with $P_G(-1)=0$ can be obtained by grafting. We present a smallest example (in terms of number of vertices) of such a connected graph below.
    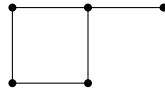
\begin{figure}[h]
    \centering
    \begin{tikzpicture}
            \coordinate (a1) at (-3,0);
			\coordinate (a2) at (-2,0);
			\coordinate (a3) at (-2,1);
            \coordinate (a4) at (-3,1);
            \coordinate (a5) at (-1,1);

			\draw (a1)--(a2)--(a3)--(a4)--(a1);
            \draw (a3)--(a5);

            \foreach \p in {a1,a2,a3,a4,a5}
			\fill (\p) circle (1.5pt);
         \end{tikzpicture}
    \caption{A graph with positive multiplicity that is not a grafting of a tree.}
    \label{fig:whisker}
\end{figure}
\end{remark}



\section{Graphs with high multiplicity}

In this section we study graphs $G$ with high values of $\mult_{-1}P_G$. The main result of this section is the following.

\begin{theorem}\label{thm:log-concave}
    Let $G$ be a finite simple graph on $n\geq 1$ vertices. If $\mult_{-1} P_G \geq \alpha(G)-2$, then $P_G$ is log-concave.
\end{theorem}





It is clear that $\mult_{-1}P_G\leq \alpha(G)$, and equality occurs exactly when $G$ is $\sqcup_{k=1}^{\alpha(G)} K_1$ (Lemma~\ref{lem:n-isolated-vertices}). For the rest of the section, if $n \geq 2$, we assume that $G$ has at least one edge. Then $\mult_{-1}P_G\leq \alpha(G)-1$. We analyze when the equality occurs in this case.


\begin{lemma}\label{lem:max-mult}
    Let $G$ be a finite simple graph on $n\geq 1$ vertices. Then $\mult_{-1}P_G= \alpha(G)-1$ if and only if 
    \[
    P_G(x)=(1+x)^{\alpha(G)-1}\left(1+(n-\alpha(G)+1)x\right),
    \]
    and $G\neq \sqcup_{k=1}^n K_1$.
    In particular, in this case, $|E(G)|= \binom{n-\alpha(G)+1}{2}$.
\end{lemma}

\begin{proof}
    Since $P_G$ is a polynomial of degree $\alpha(G)$, we have $\mult_{-1}P_G= \alpha(G)-1$ if and only if $P_G(x)=(1+x)^{\alpha(G)-1}\left(1+cx\right)$ for some constant $c$ such that $1-c\neq 0$. Moreover, we know that the coefficient of $x$ in $P_G(x)$ is $n$, which forces $c=n-\alpha+1$ in this case. Also, $1-c\neq 0$ is equivalent to $n\neq \alpha$, which occurs if and only if $G\neq \sqcup_{k=1}^n K_1$ by Lemma~\ref{lem:n-isolated-vertices}. The first statement then follows.

    For the second statement, note that the coefficient of $x^2$ in $P_G(x)$ is exactly the number of non-edges of $G$. Therefore we have
    \[
    (n-\alpha(G)+1) \binom{\alpha(G)-1}{\alpha(G)-2} + \binom{\alpha(G)-1}{2} =\binom{n}{2} -|E(G)|.
    \]
    The result then follows.
\end{proof}

Next we investigate the condition $\mult_{-1}P_G= \alpha(G)-2$. 

\begin{lemma}\label{lem:almost-max-mult}
    Let $G$ be a finite simple graph on $n\geq 1$ vertices. Then $\mult_{-1}P_G= \alpha(G)-2$ if and only if 
    \[
    P_G(x)=(1+x)^{\alpha(G)-2}\left(1+(n-\alpha(G)+2)x + \left( \binom{n-\alpha(G)+2}{2} - |E(G)| \right)x^2\right)
    \]
    and
    \[
    |E(G)|\neq \binom{n-\alpha(G)+1}{2}.
    \]
    Moreover, in this  case, we have $|E(G)| < \binom{n-\alpha(G)+2}{2}$.
\end{lemma}

\begin{proof}
    
        Since $P_G$ is a polynomial of degree $\alpha(G)$, we have $\mult_{-1}P_G= \alpha(G)-2$ if and only if $P_G(x)=(1+x)^{\alpha(G)-2}\left(1+ax+bx^2\right)$ for some constants $a,b$ with $1-a+b\neq 0$. In this case, we have
    \[
    (1+x)^{\alpha(G)-2}\left(1+ax+bx^2\right)=P_G(x)=1+nx+\left(\binom{n}{2}-|E(G)| \right) x^2 + \cdots.
    \]
    By matching the first three  coefficients, we obtain the system of equations
    \[
    \begin{cases}
        a&=n-\alpha(G)+2\\
        (\alpha(G)-2)a + b&= \binom{n}{2}-|E(G)| - \binom{\alpha(G)-2}{2}
    \end{cases}.
    \]
    The first statement then straightforwardly follows from solving this system. For the second statement, observe that by matching the leading coefficients, we have
    \[
    \binom{n-\alpha(G)+2}{2}-|E(G)| = b = \#\text{independent sets of $G$ of size $\alpha(G)$},
    \]
    which is positive by definition of $\alpha(G)$. This concludes the proof.
\end{proof}

We are now ready to prove the main result of this section

\begin{proof}[Proof of Theorem~\ref{thm:log-concave}]
    It is known that the product of two log-concave polynomials is log-concave \cite{product-unimodal}. Thus the result follows if $\mult_{-1}P_G\geq \alpha(G)-1$, since $P_G$ factors into linear forms then. In the case $\mult_{-1} P_G=\alpha(G)-2$, by Lemma~\ref{lem:almost-max-mult}, we have 
    \[
    P_G(x)=(1+x)^{\alpha(G)-2}\left(1+(n-\alpha(G)+2)x + \left( \binom{n-\alpha(G)+2}{2} - |E(G)| \right)x^2\right)
    \]
    with $\binom{n-\alpha(G)+2}{2} - |E(G)|>0$. In other words, the polynomial 
    \[
    1+(n-\alpha(G)+2)x + \left( \binom{n-\alpha(G)+2}{2} - |E(G)| \right)x^2
    \]
    has positive coefficients, and since
    \[
    (n-\alpha(G)+2)^2 >\binom{n-\alpha(G)+2}{2} \geq\binom{n-\alpha(G)+2}{2} - |E(G)|,
    \]
    it is log-concave. The result then follows.
\end{proof}

\begin{remark}
    It is tempting to obtain an analog for the next case $\mult_{-1}P_G=\alpha(G)-3$, which implies that
    \[
    P_G(x)=(1+x)^{\alpha(G)-3}(1+ax+bx^2+cx^3)
    \]
    for some integers $a,b,c$.
    It is straightfroward that
    \begin{align*}
        a&=n-\alpha(G)+3,\\
        b&=\binom{n-\alpha(G)+3}{2} - |E(G)|,\\
        c&= (\alpha(G)-3)|E(G)|+ \binom{n-\alpha(G)+3}{3} - \binom{n}{3} + \#\text{independent sets of $G$ of size 3}.
    \end{align*}
    It is unclear whether we have $b>0$. If this is true, then we know at least that $P_G$ is unimodal by similar arguments as in the proof of Theorem~\ref{thm:log-concave}, and the fact that the product of a log-concave polynomial and a unimodal one is unimodal \cite{product-unimodal}.
\end{remark}

We end this section with a natural question.

\begin{question}
    Which graphs $G$ satisfy $\mult_{-1}P_G = \alpha(G)-1$ (or $\mult_{-1}P_G = \alpha(G)-2$)?
\end{question}

Despite the restrictive conditions that $\mult_{-1}P_G \in \{ \alpha(G)-2, \alpha(G)-1\}$ imposes on the graph $G$, there are surprisingly many graphs with either value for $\mult_{-1} P_G$. We shall recall a method to obtain either class. 

\begin{definition}\label{def:whiskered}
    For a graph $G$ on $[n]$, let $\mathcal{W}(G)$ denote the graph on $[2n]$ with
\[
E(\mathcal{W}(G))=E(G) \cup \{ \{i,i+n\} \mid i\in [n] \}.
\]
\end{definition}

Pictorially, $\mathcal{W}(G)$ is exactly $G$ with a new pendant vertex each attached to vertices of $G$. The graph $\mathcal{W}(G)$ is sometimes called the \emph{whiskered graph} of $G$, and can also be obtained by the operation of corona product. We give the example of $K_4$ and $\mathcal{W}(K_4)$ below as an illustration.

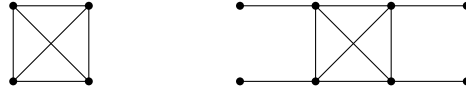
\begin{figure}[h]
    \centering
    \begin{tikzpicture}
            \coordinate (a1) at (-3,0);
			\coordinate (a2) at (-2,0);
			\coordinate (a3) at (-2,1);
            \coordinate (a4) at (-3,1);
			
			\draw (a1)--(a2)--(a3)--(a4)--(a1)--(a3);
            \draw (a2)--(a4);

            \coordinate (b1) at (1,0);
			\coordinate (b2) at (2,0);
			\coordinate (b3) at (2,1);
            \coordinate (b4) at (1,1);
            \coordinate (b5) at (0,0);
            \coordinate (b6) at (3,0);
            \coordinate (b7) at (3,1);
            \coordinate (b8) at (0,1);

			\draw (b1)--(b2)--(b3)--(b4)--(b1)--(b3);
            \draw (b2)--(b4);
            \draw (b1)--(b5);
            \draw (b2)--(b6);
            \draw (b3)--(b7);
            \draw (b4)--(b8);

            \foreach \p in {a1,a2,a3,a4,b1,b2,b3,b4,b5,b6,b7,b8}
			\fill (\p) circle (1.5pt);
         \end{tikzpicture}
    \caption{$K_4$ and $\mathcal{W}(K_4)$.}
    \label{fig:whisker}
\end{figure}

The independence polynomial of $\mathcal{W}(G)$ is well understood.

\begin{lemma}[\protect{\cite[Theorem~2.3]{MR2379079}}]\label{lem:G*}
    Let $G$ be a finite simple graph with at least one edge. Then
    \[
    \mult_{-1}P_{\mathcal{W}(G)} = |V(G)|- \alpha(G) \quad  \text{and} \quad \alpha(\mathcal{W}(G))=|V(G)|.
    \]
\end{lemma}

For example, $G=\mathcal{W}(K_n)$ satisfies $\mult_{-1}P_G = n-1 = \alpha(G)-1$. In fact we also have $\mult_{-1} P_{K_n}=0 = \alpha(K_n)-1$. We remark that there are more graphs with this property, even in small number of vertices, and it poses a challenging problem to characterize them all.



\section{Pairs of multiplicity and independence number}\label{sec:MI}

The goal of this section is to provide bounds for $\mathcal{MI}^c(n)$, together with lattice points in $\mathbb{Z}_{\geq 0}^2$ that can be realized on its boundary. It is clear that $\mathcal{MI}^c(1)=\{ (1,1) \}$. We note down the sets $\mathcal{MI}^c(n)$ for $2\leq n \leq 8$ for an illustration.

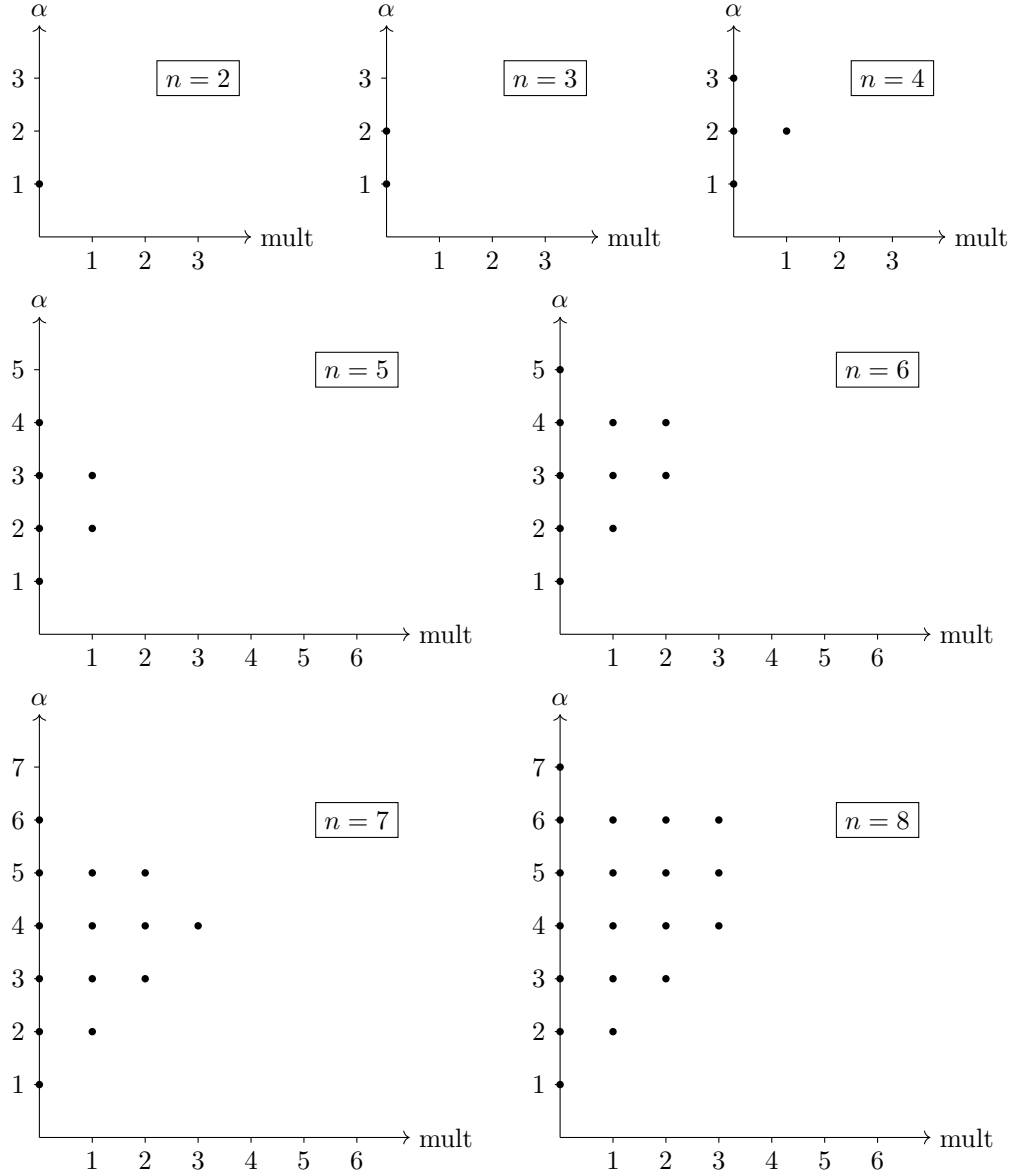
\begin{figure}[h!]
    \centering
    \begin{tabular}{cccccc}
         \multicolumn{2}{c}{\begin{tikzpicture}[scale=0.7]
    \draw[->] (0,0) -- (4,0) node[right] {$\mult$};
    \draw[->] (0,0) -- (0,4) node[above] {$\alpha$};

    \foreach \x in {1,2,3}
        \draw (\x,0) -- (\x,-0.1) node[below] {\x};

    \foreach \y in {1,2,3}
        \draw (0,\y) -- (-0.1,\y) node[left] {\y};

    \fill (0,1) circle (2pt);
    \node[draw, rectangle] at (3,3) {$n=2$};
\end{tikzpicture}}
&\multicolumn{2}{c}{\begin{tikzpicture}[scale=0.7]
    \draw[->] (0,0) -- (4,0) node[right] {$\mult$};
    \draw[->] (0,0) -- (0,4) node[above] {$\alpha$};

    \foreach \x in {1,2,3}
        \draw (\x,0) -- (\x,-0.1) node[below] {\x};

    \foreach \y in {1,2,3}
        \draw (0,\y) -- (-0.1,\y) node[left] {\y};

    \fill (0,1) circle (2pt);
    \fill (0,2) circle (2pt);
    \node[draw, rectangle] at (3,3) {$n=3$};
\end{tikzpicture}}
& \multicolumn{2}{c}{  \begin{tikzpicture}[scale=0.7]
    \draw[->] (0,0) -- (4,0) node[right] {$\mult$};
    \draw[->] (0,0) -- (0,4) node[above] {$\alpha$};

    \foreach \x in {1,2,3}
        \draw (\x,0) -- (\x,-0.1) node[below] {\x};

    \foreach \y in {1,2,3}
        \draw (0,\y) -- (-0.1,\y) node[left] {\y};

    \fill (0,1) circle (2pt);
    \fill (0,2) circle (2pt);
    \fill (0,3) circle (2pt);
    \fill (1,2) circle (2pt);
    \node[draw, rectangle] at (3,3) {$n=4$};
\end{tikzpicture}}\\
    \multicolumn{3}{c}{\begin{tikzpicture}[scale=0.7]
    \draw[->] (0,0) -- (7,0) node[right] {$\mult$};
    \draw[->] (0,0) -- (0,6) node[above] {$\alpha$};

    \foreach \x in {1,2,3,4,5,6}
        \draw (\x,0) -- (\x,-0.1) node[below] {\x};

    \foreach \y in {1,2,3,4,5}
        \draw (0,\y) -- (-0.1,\y) node[left] {\y};

    \fill (0,1) circle (2pt);
    \fill (0,2) circle (2pt);
    \fill (0,3) circle (2pt);
    \fill (0,4) circle (2pt);
    \fill (1,2) circle (2pt);
    \fill (1,3) circle (2pt);
    \node[draw, rectangle] at (6,5) {$n=5$};
\end{tikzpicture}} & \multicolumn{3}{c}{    \begin{tikzpicture}[scale=0.7]
    \draw[->] (0,0) -- (7,0) node[right] {$\mult$};
    \draw[->] (0,0) -- (0,6) node[above] {$\alpha$};

    \foreach \x in {1,2,3,4,5,6}
        \draw (\x,0) -- (\x,-0.1) node[below] {\x};

    \foreach \y in {1,2,3,4,5}
        \draw (0,\y) -- (-0.1,\y) node[left] {\y};

    \fill (0,1) circle (2pt);
    \fill (0,2) circle (2pt);
    \fill (0,3) circle (2pt);
    \fill (0,4) circle (2pt);
    \fill (0,5) circle (2pt);
    \fill (1,2) circle (2pt);
    \fill (1,3) circle (2pt);
    \fill (1,4) circle (2pt);
    \fill (2,3) circle (2pt);
    \fill (2,4) circle (2pt);
    \node[draw, rectangle] at (6,5) {$n=6$};
\end{tikzpicture}}\\
        \multicolumn{3}{c}{\begin{tikzpicture}[scale=0.7]
    \draw[->] (0,0) -- (7,0) node[right] {$\mult$};
    \draw[->] (0,0) -- (0,8) node[above] {$\alpha$};

    \foreach \x in {1,2,3,4,5,6}
        \draw (\x,0) -- (\x,-0.1) node[below] {\x};

    \foreach \y in {1,2,3,4,5,6,7}
        \draw (0,\y) -- (-0.1,\y) node[left] {\y};

    \fill (0,1) circle (2pt);
    \fill (0,2) circle (2pt);
    \fill (0,3) circle (2pt);
    \fill (0,4) circle (2pt);
    \fill (0,5) circle (2pt);
    \fill (0,6) circle (2pt);
    \fill (1,2) circle (2pt);
    \fill (1,3) circle (2pt);
    \fill (1,4) circle (2pt);
    \fill (1,5) circle (2pt);
    \fill (2,3) circle (2pt);
    \fill (2,4) circle (2pt);
    \fill (2,5) circle (2pt);
    \fill (3,4) circle (2pt);
    \node[draw, rectangle] at (6,6) {$n=7$};
\end{tikzpicture}} & \multicolumn{3}{c}{\begin{tikzpicture}[scale=0.7]
    \draw[->] (0,0) -- (7,0) node[right] {$\mult$};
    \draw[->] (0,0) -- (0,8) node[above] {$\alpha$};

    \foreach \x in {1,2,3,4,5,6}
        \draw (\x,0) -- (\x,-0.1) node[below] {\x};

    \foreach \y in {1,2,3,4,5,6,7}
        \draw (0,\y) -- (-0.1,\y) node[left] {\y};

    \fill (0,1) circle (2pt);
    \fill (0,2) circle (2pt);
    \fill (0,3) circle (2pt);
    \fill (0,4) circle (2pt);
    \fill (0,5) circle (2pt);
    \fill (0,6) circle (2pt);
    \fill (0,7) circle (2pt);
    \fill (1,2) circle (2pt);
    \fill (1,3) circle (2pt);
    \fill (1,4) circle (2pt);
    \fill (1,5) circle (2pt);
    \fill (1,6) circle (2pt);
    \fill (2,3) circle (2pt);
    \fill (2,4) circle (2pt);
    \fill (2,5) circle (2pt);
    \fill (2,6) circle (2pt);
    \fill (3,4) circle (2pt);
    \fill (3,5) circle (2pt);
    \fill (3,6) circle (2pt);
    \node[draw, rectangle] at (6,6) {$n=8$};
\end{tikzpicture}}
    \end{tabular}
    \caption{The set $\mathcal{MI}^c(n)$ for $n=2,3,4,5,6,7,8$.}
    \label{fig:MI}
\end{figure} 

For the rest of this section, we assume that $n\geq 2$. We start with a straightforward upper bound for $\mathcal{MI}^c(n)$.

\begin{lemma}\label{lem:MI-n-connected-upperbound}
    Let $G$ be a connected graph on $n\geq 2$ vertices. Then $0\leq \mult_{-1}P_G < \alpha(G)\leq n-1$. In other words,
    \[
     \mathcal{MI}^c(n) \subseteq  \{(a,b)\mid 0\leq a<b\leq n-2 \} \cup \{ (0,n-1) \}.
    \]
\end{lemma}
\begin{proof}
    It is straightforward that
    \[
    0\leq \mult_{-1}P_G \leq  \alpha(G)\leq n.
    \]
    However, either $\alpha(G)=n$ or $\mult_{-1}P_G = \alpha(G)$ would imply that $G$ is the graph of $n\geq 2$ isolated vertices (Lemma~\ref{lem:n-isolated-vertices}), a contradiction. On the other hand, by Lemma~\ref{lem:star-graph}, we have $\alpha(G)=n-1$ if and only if $G$ is the star graph $S_n$, which in turn implies that $\mult_{-1}P_G=0$.  The result then follows. 
\end{proof}

The next goal is to present a lower bound for $\mathcal{MI}^c(n)$. To do so, it is necessary to present graph operations where we can control both the multiplicity and independence number of a graph. The cone operation is one such operation. The following is a direct translation of Lemma~\ref{lem:cone-ind-poly}. Thus we do not provide a proof.

\begin{lemma}\label{lem:cone-mult-alpha}
    Let $n\geq 1$ be an integer and $(a,b)\in \mathcal{MI}^c(n)$. Then we have the following:
    \begin{enumerate}
        \item $(a,b+1)\in \mathcal{MI}^c(n+1)$;
        \item $(a,b) \in \mathcal{MI}^c(n+1)$ if $a+2\leq b$;
        \item $(a',b)\in \mathcal{MI}^c(n+1)$ for any $0\leq a'<a$.
    \end{enumerate}
\end{lemma}

A shortcoming of the above lemma is that it does not address the points $(a,a+1)\in \mathcal{MI}^c(n)$. We give a positive answer to this in the next result.

\begin{lemma}\label{lem:diagonal-realization-subset}
    Let $n\geq 2$ be an integer and $(a,a+1)\in \mathcal{MI}^c(n)$ for some $a\geq 0$. Then $(a,a+1)\in \mathcal{MI}^c(n+1)$.
\end{lemma}
\begin{proof}
    Since $(0,1)$ is realized by the complete graph $K_n$ (Lemma~\ref{lem:complete-graph}) for any $n\geq 2$, we have $(0,1)\in \mathcal{MI}^c(n)$ for any $n\geq 2$. Thus for the rest of the proof we can assume that $a\geq 1$.
    
    Let $G$ be a connected graph on $n$ vertices with $\mult_{-1}P_G = a$ and $\alpha(G)=a+1$. The existence of $G$ is guaranteed by the hypothesis $(a,a+1)\in \mathcal{MI}^c(n)$. Then by Lemma~\ref{lem:max-mult}, we have
    \[
    P_G(x)=(1+x)^{a}(1+(n-a)x).
    \]    
    Since we have $\alpha(G)=a+1$, there exists an independent set $U$ of $G$ such that $|U|=a$. By Lemma~\ref{lem:cone-ind-poly}, we have
    \[
    P_{\cone(G,U)} (x)=P_G(x)+x(1+x)^a = (1+x)^{a}(1+(n-a+1)x).
    \]
    Remark that $1+(n-a+1)(-1) = \alpha(G)-n-1 \leq -1$. Thus 
    \[
    \mult_{-1} P_{\cone(G,U)}  = a \quad \text{and} \quad \alpha(\cone(G,U) )=a+1.
    \]
    Finally, due to $a\geq 1$, the graph $\cone(G,U)$ is connected on $n+1$ vertices. Thus $(a,a+1)\in \mathcal{MI}^c(n+1)$, as desired.
\end{proof}

As a consequence, we show that the set $\mathcal{MI}^c(n)$ does become larger as $n$ grows.

\begin{theorem}\label{thm:MI-c-increase}
    For any $n\geq 2$, we have
    \[
    \mathcal{MI}^c(n)\subseteq \mathcal{MI}^c(n+1).
    \]
\end{theorem}

\begin{proof}
    Let $(a,b)\in \mathcal{MI}^c(n)$. By Lemma~\ref{lem:MI-n-connected-upperbound}, we have $a<b$. If $a+2\leq b$, then the result follows from Lemma~\ref{lem:cone-mult-alpha} (2). On the other hand, if $b=a+1$, then the result follows from Lemma~\ref{lem:diagonal-realization-subset}. This concludes the proof.
\end{proof}

We shall construct graphs on the line $\mult_{-1} P_G=\lceil n/2\rceil -1$.

For each $r,s\geq 1$, let $T_{r,s}$ and $T_{r}$ denote the graphs where the vertex sets are
\[
V(T_{r,s}) = [r+s+2] \quad \text{and}\quad V(T_r)= [2r+5]
\]
and the edges sets are
\begin{multline*}
    E(T_{r,s}) = \{ \{1,2\}, \{1,i\}, \{2,j+r\} \mid 3\leq i \leq r+2 \text{ and } 3\leq j\leq s+2  \} \quad \text{and} \\
    E(T_r) = \{ \{1,2\}, \{2,3\}, \{3,4\}, \{3,5\}, \{3,i\}, \{5,i+r\} \mid 6\leq i\leq r+5 \}.
\end{multline*}

Pictorially, $T_{r,s}$ is a tree of diameter 3, while $T_r$ is a tree of diameter 4. We illustrate these graphs with some pictures below.

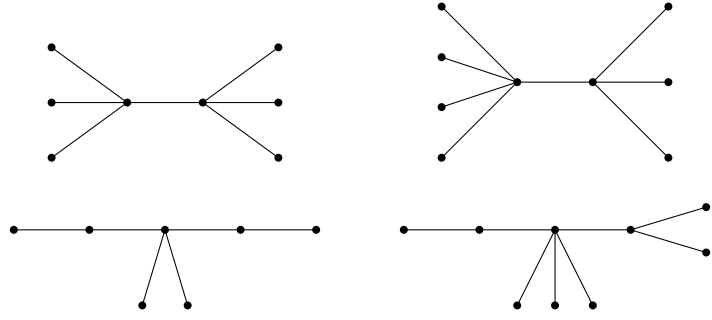
\begin{figure}[h!]
    \centering
    \begin{tabular}{cccc}
         \begin{tikzpicture}
             \coordinate (v1) at (0,0);
		\coordinate (v2) at (1,0);
		\coordinate (v3) at (-1,0.73);
		\coordinate (v4) at (-1,0);
		\coordinate (v5) at (-1,-0.73);
		\coordinate (v6) at (2,0.73);
		\coordinate (v7) at (2,0);
		\coordinate (v8) at (2,-0.73);

\fill (v1) circle (1.5pt);
		\fill (v2) circle (1.5pt);
		\fill (v3) circle (1.5pt);
		\fill (v4) circle (1.5pt);
		\fill (v5) circle (1.5pt);
		\fill (v6) circle (1.5pt);
        \fill (v7) circle (1.5pt);
        \fill (v8) circle (1.5pt);

        \draw (v1)--(v2);
        \draw (v1)--(v3);
        \draw (v1)--(v4);
        \draw (v1)--(v5);
        \draw (v2)--(v6);
        \draw (v2)--(v7);
        \draw (v2)--(v8);
\end{tikzpicture}& && \begin{tikzpicture}
             \coordinate (v1) at (0,0);
		\coordinate (v2) at (1,0);
		\coordinate (v3) at (-1,1);
		\coordinate (v4) at (-1,0.33);
		\coordinate (v5) at (-1,-0.33);
		\coordinate (v6) at (-1,-1);
		\coordinate (v7) at (2,1);
		\coordinate (v8) at (2,0);
		\coordinate (v9) at (2,-1);

\fill (v1) circle (1.5pt);
		\fill (v2) circle (1.5pt);
		\fill (v3) circle (1.5pt);
		\fill (v4) circle (1.5pt);
		\fill (v5) circle (1.5pt);
		\fill (v6) circle (1.5pt);
        \fill (v7) circle (1.5pt);
        \fill (v8) circle (1.5pt);
        \fill (v9) circle (1.5pt);

        \draw (v1)--(v2);
        \draw (v1)--(v3);
        \draw (v1)--(v4);
        \draw (v1)--(v5);
        \draw (v1)--(v6);
        \draw (v2)--(v7);
        \draw (v2)--(v8);
        \draw (v2)--(v9);
\end{tikzpicture}\\
&&&\\

    \begin{tikzpicture}
             \coordinate (v1) at (0,0);
		\coordinate (v2) at (1,0);
		\coordinate (v3) at (2,0);
		\coordinate (v4) at (3,0);
		\coordinate (v5) at (4,0);
		\coordinate (v6) at (2.3,-1);
		\coordinate (v7) at (1.7,-1);

\fill (v1) circle (1.5pt);
		\fill (v2) circle (1.5pt);
		\fill (v3) circle (1.5pt);
		\fill (v4) circle (1.5pt);
		\fill (v5) circle (1.5pt);
		\fill (v6) circle (1.5pt);
        \fill (v7) circle (1.5pt);

        \draw (v1)--(v2)--(v3)--(v4)--(v5);
        \draw (v3)--(v6);
        \draw (v3)--(v7);
\end{tikzpicture}& &&\begin{tikzpicture}
             \coordinate (v1) at (0,0);
		\coordinate (v2) at (1,0);
		\coordinate (v3) at (2,0);
		\coordinate (v4) at (3,0);
		\coordinate (v5) at (4,0.3);
        \coordinate (v6) at (4,-0.3);
		\coordinate (v7) at (2.5,-1);
		\coordinate (v8) at (1.5,-1);
        \coordinate (v9) at (2,-1);

\fill (v1) circle (1.5pt);
		\fill (v2) circle (1.5pt);
		\fill (v3) circle (1.5pt);
		\fill (v4) circle (1.5pt);
		\fill (v5) circle (1.5pt);
		\fill (v6) circle (1.5pt);
        \fill (v7) circle (1.5pt);
        \fill (v8) circle (1.5pt);
        \fill (v9) circle (1.5pt);

        \draw (v1)--(v2)--(v3)--(v4)--(v5);
        \draw (v4)--(v6);
        \draw (v3)--(v7);
        \draw (v3)--(v8);
        \draw (v3)--(v9);
\end{tikzpicture}
    \end{tabular}
    \caption{The graphs $T_{3,3}$, $T_{4,3}$, $T_1$, and $T_2$.}
    \label{fig:trees}
\end{figure}

A formula for the independence polynomial for a tree of diameter at most 4 has been obtained in \cite[Proposition~5.5]{biermann-et-al.2026}, and thus both the multiplicity and independence number of such a graph are known. We record the formulae for our specially constructed graphs below.

\begin{lemma}[\protect{\cite[Remark 5.5 and Lemma 5.7]{biermann-et-al.2026}}]\label{lem:special-trees}
    For each $r,s\geq 1$, we have
    \[
    \mult_{-1} P_{T_{r,s}} =\min\{r,s\}  \quad \text{and} \quad \alpha(T_{r,s})=r+s,
    \]
    and
    \[
    \mult_{-1} P_{T_r} =r +2 \quad \text{and} \quad \alpha(T_r)=2r+2.
    \]
\end{lemma}

We will also need the following lemma. 

\begin{lemma}\label{lem:a+1-and-b+1}
    Let $n\geq 2$ be an integer and $(a,b)\in \mathcal{MI}^c(n)$ for some $a\geq 0$. Then $(a+1,b+1)\in \mathcal{MI}^c(n+2)$.
\end{lemma}

\begin{proof}
    Let $G$ be a connected graph with $n$ vertices that realizes the pair $(a,b)$, i.e., $\mult_{-1} P_G=a$ and $\alpha(G)=b$. The latter implies that there exists an independent set $U$ of $G$ of size $b$. Note that $U$ is also an independent set of $G\sqcup K_1$. We then consider the graph $\cone(G\sqcup K_1,U)$, a connected graph with $n+2$ vertices by construction. It now suffices to show that
    \begin{equation}\label{eq:a+1-and-b+1}
        \mult_{-1} P_{\cone(G\sqcup K_1,U)} = a+1 \quad \text{and} \quad \alpha(\cone(G\sqcup K_1,U)) = b+1.
    \end{equation}

    Assume that $b\geq a+2$. Then by Lemma~\ref{lem:union-graphs}, we have
    \[
    \mult_{-1} P_{G\sqcup K_1} = a+1 < b< b+1=\alpha(G\sqcup K_1).
    \]
    Thus (\ref{eq:a+1-and-b+1}) follows from Lemma~\ref{lem:cone-ind-poly}. 

    Now we can assume that $b=a+1$ since $b>a$ by Lemma~\ref{lem:MI-n-connected-upperbound}. By Lemmas~\ref{lem:union-graphs} and \ref{lem:max-mult}, we have
    \[
    P_{G\sqcup K_1}(x) = (1+x)^{b}(1+(n-b+1)x). 
    \]
    Thus by Lemma~\ref{lem:cone-ind-poly}, we obtain
    \[
    P_{\cone(G\sqcup K_1,U)}(x)= (1+x)^{b}(1+(n-b+1)x) + x(1+x)^b = (1+x)^b(1+(n-b+2)x),
    \]
    which implies $\alpha(\cone(G\sqcup K_1,U))=b+1$. On the other hand,
    since
    \[
    1+(n-b+2)(-1) = \alpha(G) -n-1 \leq -1,
    \]
    we have $\mult_{-1} P_{\cone(G\sqcup K_1,U)} = b=a+1$. Thus (\ref{eq:a+1-and-b+1}) holds, as desired.
\end{proof}

We are now ready to give a lower bound for $\mathcal{MI}^c(n)$.

\begin{theorem}\label{thm:lower-bound-MI-c}
    Let $n\geq 2$ be an integer. If $n$ is odd, then
    \begin{equation}\label{eq:MI-c-odd}
        \{(a,b)\in \mathbb{Z}_{\geq 0}^2\mid 0\leq a<b\leq n-2 \text{ and } a\leq \lceil n/2 \rceil-1 \} \setminus \{(\lceil n/2  \rceil-1,n-2)\}\cup \{(0,n-1)\}\subseteq \mathcal{MI}^c(n),
    \end{equation}
    and if $n$ is even, then 
    \begin{equation}\label{eq:MI-c-even}
        \{(a,b)\in \mathbb{Z}_{\geq 0}^2\mid 0\leq a<b\leq n-2 \text{ and } a\leq  n/2 -1 \} \cup \{(0,n-1)\}\subseteq \mathcal{MI}^c(n).
    \end{equation}
\end{theorem}

\begin{proof}
    Note that the results follow from Figure~\ref{fig:MI} if $n\leq 8$. We now proceed by induction on $n$.

    Assume that $n\geq 8$ is even. By induction, (\ref{eq:MI-c-even}) holds for $n-2$: 
    \begin{align*}
        \{(a,b)\in \mathbb{Z}_{\geq 0}^2\mid 0\leq a<b\leq n-4 \text{ and } a\leq \lceil (n-2)/2 \rceil-1 \} \cup \{(0,n-3)\} \subseteq \mathcal{MI}^c(n-2).
    \end{align*}
    Note that $\lceil (n-2)/2 \rceil = n/2 -1$ since $n$ is even. We can then rewrite the above as follows:
    \[
    \{(a,b)\in \mathbb{Z}_{\geq 0}^2\mid 0\leq a<b\leq n-4 \text{ and } a\leq  n/2 -2 \}  \cup \{(0,n-3)\}\subseteq \mathcal{MI}^c(n-2).
    \]
    By Lemmas~\ref{lem:cone-mult-alpha} (1) and \ref{lem:a+1-and-b+1} and Theorem~\ref{thm:MI-c-increase}, we have 
    \begin{align*}
        \mathcal{MI}^c(n)
        &\supseteq \{(a,b),(a,b+1), (a,b+2), (a+1,b+1)\in \mathbb{Z}_{\geq 0}^2\mid (a,b)\in \mathcal{MI}^c(n-2)\}\\
        &\supseteq \{(a,b)\in \mathbb{Z}_{\geq 0}^2\mid 0\leq a<b\leq n-2 \text{ and } a\leq n/2 -1 \}  \cup \{(0,n-1)\}\setminus \{(n/2-1,n-2)\}.
    \end{align*}
    On the other hand, the pair $(n/2-1,n-2)$ can be realized by the graph $T_{n/2-1,n/2-1}$, a tree with $n$ vertices, by Lemma~\ref{lem:special-trees}. Thus (\ref{eq:MI-c-even}) holds, as desired.
    
    Now we can assume that $n\geq 9$ is odd. By induction, (\ref{eq:MI-c-odd}) holds for $n-2$:
    \begin{multline*}
        \{(a,b)\in \mathbb{Z}_{\geq 0}^2\mid 0\leq a<b\leq n-4 \text{ and } a\leq \lceil (n-2)/2 \rceil-1 \} \setminus \{(\lceil (n-2)/2  \rceil-1,n-4)\}\\
        \cup \{(0,n-3)\}\subseteq \mathcal{MI}^c(n-2)
    \end{multline*}
    Note that $\lceil (n-2)/2 \rceil =  (n-1)/2 $ since $n$ is odd. We can then rewrite the above as follows:
    \begin{multline*}
        \{(a,b)\in \mathbb{Z}_{\geq 0}^2\mid 0\leq a<b\leq n-4 \text{ and } a\leq (n-3)/2  \} \setminus \{((n-3)/2 ,n-4)\}\\
        \cup \{(0,n-3)\}\subseteq \mathcal{MI}^c(n-2).
    \end{multline*}
    Note that we have $((n-3)/2,n-5)\in \mathcal{MI}^c(n-2)$ since $n\geq 9$.
    By Lemmas~\ref{lem:cone-mult-alpha} (1) and \ref{lem:a+1-and-b+1} and Theorem~\ref{thm:MI-c-increase}, we then have 
    \begin{align*}
        \mathcal{MI}^c(n)
        &\supseteq \{(a,b),(a,b+1), (a,b+2), (a+1,b+1)\in \mathbb{Z}_{\geq 0}^2\mid (a,b)\in \mathcal{MI}^c(n-2)\}\\
        &\begin{multlined}[t]
            \supseteq \{(a,b)\in \mathbb{Z}_{\geq 0}^2\mid 0\leq a<b\leq n-2 \text{ and } a\leq (n-1)/2  \}
        \cup \{(0,n-1)\}\\
         \setminus \{((n-1)/2 ,n-2),  ((n-3)/2 ,n-2), ((n-1)/2 ,n-3)\}.
        \end{multlined}
    \end{align*}
    On the other hand, the pairs  $  ((n-3)/2 ,n-2)$ and $((n-1)/2 ,n-3)$ can be realized by the trees $T_{(n-3)/2,(n-1)/2}$ and $T_{(n-5)/2}$, respectively, with both on $n$ vertices, by Lemma~\ref{lem:special-trees}. Thus (\ref{eq:MI-c-odd}) holds, as desired.
\end{proof}




In fact, the lower bound in Theorem~\ref{thm:lower-bound-MI-c} is indeed the whole set $\mathcal{MI}^c(n)$ for $2\leq n\leq 8$. We believe this to be always the case. It is straightforward to see that the equality reduces to the following question.

\begin{question}\label{ques}
    Given a connected graph $G$ on $n\geq 2$ vertices. Is it true that $\mult_{-1}P_G\leq \lceil n/2\rceil -1$?
\end{question}

We fell short of determining  $\mathcal{MI}^c(n)$. Instead we will determine the realizable points along its boundary, which according to Lemma~\ref{lem:MI-n-connected-upperbound} is determined by the three lines $a=0$, $b=n-2$, and $b=a+1$ in the plane $\mathbb{Z}^2_{\geq 0}$. We start with the horizontal line $b=n-2$.

\begin{theorem}\label{thm:alpha=n-2}
    Let $n\geq 2$ be an integer. Then
    \[
    \mathcal{MI}^c(n) \cap \{(a,n-2)\in \mathbb{Z}_{\geq 0}^2\} = \{ (a,n-2) \in \mathbb{Z}_{\geq 0}^2\mid 0\leq a\leq \lfloor n/2 \rfloor -1 \}.
    \]
\end{theorem}
\begin{proof}
    $(\subseteq):$ Let $G$ be a graph $n$ vertices with $\alpha(G)=n-2$. We can then set $V(G)=[n]$ such that $[n-2]$ is an independent set of $G$. We want to show that $\mult_{-1}P_G \leq \lfloor n/2 \rfloor-1$. Set
    \begin{align*}
        r&= |\{ i\in [n-2] \mid \text{$i$ is adjacent to $n-1$, but not $n$} \}|  &= |N_G(n-1)\setminus N_G(n)|,\\
        s&= |\{ i\in [n-2] \mid \text{$i$ is adjacent to $n$, but not $n-1$} \}|&=|N_G(n)\setminus N_G(n-1)|,\\
        t&= |\{ i\in [n-2] \mid \text{$i$ is adjacent to $n-1$ and $n$} \}|&=|N_G(n-1)\cap N_G(n)|.
    \end{align*}
    In particular we have $r+s+t=n-2$. Applying Lemma~\ref{lem:basic-identities}, we have
    \begin{equation}\label{eq:ind-poly-alpha=n-2}
        P_G(x)= P_{G\setminus \{n\}} (x) + xP_{G\setminus N[n]}(x).
    \end{equation}
    We have two cases.

    \textbf{Case 1:} Assume that $\{n-1,n\}\notin E(G)$. Then
     $G\setminus \{n\}$ is the disjoint union of the star graph $S_{r+t+1}$ with $s$ isolated vertices, and that $G\setminus N[n]$ is the star graph $S_{r+1}$. Thus (\ref{eq:ind-poly-alpha=n-2}) gives
    \[
    P_G(x)= (1+x)^s \left( (1+x)^{r+t} +x \right) + x\left((1+x)^{r}+x\right) = (1+x)^{n-2} + x(1+x)^r + x(1+x)^s +x^2.
    \]
    Note that $n\geq 3$, as $n=2$ would imply that $G$ is the graph of two isolated vertices, a contradiction to the hypothesis that $G$ is connected. Thus
    \[
    \mult_{-1}P_G = \begin{cases}
        0 &\text{if $r=s=0$ or $r,s>0$},\\
        1 &\text{if exactly one between $r$ and $s$ is 0}.
    \end{cases}
    \]
    If $n=3$, then $G$ being connected forces $r=s=0$ and $t=1$. Then $\mult_{-1}P_G=\mult_{-1}P_{S_3}=0 \leq \lfloor 3/2\rfloor -1$, as desired. On the other hand, if $n\geq 4$, then $\mult_{-1} P_G \leq 1 \leq \lfloor n/2\rfloor -1$, as~desired.

    \textbf{Case 2:} Assume that $\{n-1,n\}\in E(G)$. Then
     $G\setminus \{n\}$ is the disjoint union of the star graph $S_{r+t+1}$ with $s$ isolated vertices, and that $G\setminus N[n]$ is the graph of $r$ isolated vertices. Thus (\ref{eq:ind-poly-alpha=n-2}) gives
     \[
     P_G(x)= (1+x)^s \left( (1+x)^{r+t} +x \right) + x(1+x)^{r} = (1+x)^{n-2} + x(1+x)^r+x(1+x)^s. 
     \]
     Thus \[
     \mult_{-1}P_G = \min \{r,s\}\leq \left\lfloor \frac{n-2}{2} \right\rfloor =\left\lfloor \frac{n}{2} \right\rfloor-1,
     \]
     as desired, where the inequality is due to the condition $r+s+t=n-2$.

    $(\supseteq):$ This follows from Theorem~\ref{thm:lower-bound-MI-c}.
\end{proof}

All the possible points on the vertical line $a=0$ can be realized.

\begin{corollary}\label{cor:the-line-a=0}
    Let $n\geq 2$ be an integer. Then
    \[
    \mathcal{MI}^c(n) \cap \{(0,b)\in \mathbb{Z}_{\geq 0}^2\} = \{ (0,b) \in \mathbb{Z}_{\geq 0}^2\mid 0<b\leq n -1 \}.
    \]
\end{corollary}

\begin{proof}
    The inclusion $(\subseteq)$ is from Theorem~\ref{lem:MI-n-connected-upperbound}, while $(\supseteq)$ follows from Theorem~\ref{thm:lower-bound-MI-c}. 
\end{proof}

The line $b=a+1$ is a lot trickier. Intuitively, the points $(a,a+1)$ with low values of $a$ can be realized. This is supported by Theorem~\ref{thm:lower-bound-MI-c}. The question becomes how large $a$ can be, in terms of a function in $n$, provided that $(a,a+1)$ can be realized in $\mathcal{MI}^c(n)$. 

\begin{proposition}\label{prop:diagonal-upper-bound}
    For any $n\geq 2$, if $(a,a+1)$ can be realized in $\mathcal{MI}^c(n)$ for some integer $a$, then
    \[
    a\leq n-\frac{1+\sqrt{8n-7}}{2},
    \]
    with equality implying that the graph that realizes $(a,a+1)$ is a tree.
\end{proposition}
\begin{proof}
    Let $G$ be a connected graph on $n$ vertices with $\mult_{-1}P_G=a$ and $\alpha(G)=a+1$. Since $G$ is connected, it has at least $n-1$ edges. By Lemma~\ref{lem:max-mult}, we have
    \[
    \binom{n-\alpha(G)+1}{2} \geq n-1.
    \]
    Solving this inequality, we obtain
    \[
    \text{either} \quad \alpha(G) \geq \frac{2n+1+\sqrt{8n-7}}{2} \quad \text{or} \quad \alpha(G) \leq \frac{2n+1-\sqrt{8n-7}}{2}.
    \]
    The former is not possible since it would imply that $\alpha(G) \geq \frac{2n+1+\sqrt{8n-7}}{2} > n$. We thus obtain
    \[
    \alpha(G) \leq \frac{2n+1-\sqrt{8n-7}}{2}.
    \]
    Substituting $\alpha(G)=a+1$, we obtain the desired result.
\end{proof}

\begin{remark}
    Unfortunately, not all lattice points $(a,a+1)$ with $a\leq n-\frac{1+\sqrt{8n-7}}{2}$ can be realized in $\mathcal{MI}^c(n)$. Let $n=\frac{k^2+7}{8}$ where $k$ is an odd positive integer. We then have\[
    n-\frac{1+\sqrt{8n-7}}{2} = \frac{k^2-4k+3}{2}.
    \] Consider the point 
    \[
    A_k=\left( \frac{k^2-4k+3}{2}, \frac{k^2-4k+5}{2}\right).
    \]
    By Proposition~\ref{prop:diagonal-upper-bound}, if a connected graph on $n=\frac{k^2+7}{8}$ vertices realizes $A_k$, it must be a tree. Experiments show that $A_3, A_5, A_7$ can be realized, while $A_1, A_9, A_{11}$ cannot, a somewhat surprising~observation.
    \begin{figure}[h!]
        \centering
        \begin{tabular}{ccccc}
             \begin{tikzpicture}[scale=1, every node/.style={circle,fill=black,inner sep=1pt}]
\node (A) at (1,3) {};
\node (B) at (0,3) {};

\draw (A)--(B);
\end{tikzpicture}&& \begin{tikzpicture}[scale=1, every node/.style={circle,draw,fill=black,inner sep=1pt}]
\node (A1) at (0,0) {};
\node (A2) at (1,0) {};
\node (A3) at (2,0) {};
\node (A4) at (3,0) {};

\draw (A1)--(A2)--(A3)--(A4);
\end{tikzpicture} && \begin{tikzpicture}[scale=1, every node/.style={circle,draw,fill=black,inner sep=1pt}]
\node (A1) at (0,0) {};
\node (A2) at (1,0) {};
\node (A3) at (2,0) {};
\node (A4) at (3,0) {};
\node (A5) at (4,0) {};
\node (B1) at (2.5,1) {};
\node (B2) at (1.5,1) {};

\draw (A1)--(A2)--(A3)--(A4)--(A5);
\draw (B2)--(A3)--(B1);
\end{tikzpicture}
        \end{tabular}
        \caption{Trees on $n=\frac{k^2+7}{8}$ vertices that realize $A_k$ for $k=3,5,7$.}
        \label{fig:k=3,5,7}
    \end{figure}
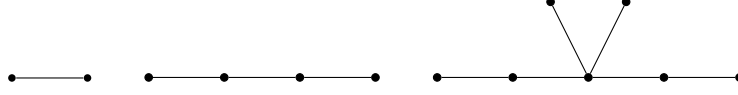
\end{remark}

All pairs that can be realized by connected graphs, without the restriction on the number of vertices, can be determined easily.

\begin{corollary}\label{cor:MI-all-connected}
    We have
    \[
    \bigcup_{n=1}^\infty \mathcal{MI}^c(n) = \{(a,b) \in \mathbb{Z}_{\geq 0}^2\mid 0\leq a< b \} \cup \{(1,1)\}.
    \]
\end{corollary}

\begin{proof}
    The containment $(\subseteq)$ follows from Lemma~\ref{lem:MI-n-connected-upperbound} and the fact that $\mathcal{MI}^c( 1) = \{(1,1)\}$, while $(\supseteq)$ follows from Theorem~\ref{thm:lower-bound-MI-c}. 
\end{proof}

Finally, we compute $\mathcal{MI}(n)$ for any $n\geq 1$. Allowing disconnected graphs essentially means that the union operation can be used in graph construction.

\begin{lemma}\label{lem:add-1-isolated-vertex}
    If $(a,b)\in \mathcal{MI}(n)$ for some integers $a,b,n$, then $(a+1,b+1)\in \mathcal{MI}(n+1)$.
\end{lemma}
\begin{proof}
    If $G$ is a graph on $n$ vertices with $\mult_{-1}P_G=a$ and $\alpha(G)=b$, then $\mult_{-1}P_{G\sqcup K_1}=a+1$ and $\alpha(G\sqcup K_1)=b+1$ by Lemma~\ref{lem:union-graphs}. The result then follows.
\end{proof}

\begin{theorem}\label{thm:MI-n-all-graphs}
    Let $n\geq 1$ be an integer. We have
    \[
    \mathcal{MI}(n) = \{(a,b)\in \mathbb{Z}_{\geq 0}^2\mid 0\leq a< b \leq n-1 \} \cup \{(n,n)\}.
    \]
\end{theorem}
\begin{proof}
    $(\subseteq):$ It is clear that $0\leq \mult_{-1} P_G\leq \alpha(G) \leq n$ for any graph $G$ on $n$ vertices. However, note that either $\mult_{-1}P_G= \alpha(G)$ or $\alpha(G)=n$ implies that $\mult_{-1}P_G= \alpha(G)=n$ by Lemma~\ref{lem:n-isolated-vertices}. The inclusion $(\subseteq)$ then follows.

    $(\supseteq):$ It is straightforward that $\mathcal{MI}(1)=\{(1,1)\}$ and $\mathcal{MI}(2)=\{(0,1),(2,2)\}$. These settle the result in the cases $n\in \{1,2\}$. By induction, we assume that 
    \[
    \mathcal{MI}(n-1) = \{(a,b)\mid 0\leq a< b \leq n-2 \} \cup \{(n-1,n-1)\}
    \]
    for some $n\geq 3$. Then by Lemma~\ref{lem:add-1-isolated-vertex}, we have
    \[
    \{(a,b)\mid 1\leq a< b \leq n-1 \} \cup \{(n,n)\} \subseteq \mathcal{MI}(n).
    \]
    It now suffices to realize the pairs $(0,b)$ where $1\leq b\leq n-1$ with graphs on $n$ vertices. If $b=1$, then the complete graph $K_n$ realizes $(0,b)=(0,1)$ by Lemma~\ref{lem:complete-graph}. If $b=n-1$, then the star graph $S_n$ realizes the pair $(0,b)=(0,n-1)$ by Lemma~\ref{lem:star-graph}. Now we can assume that $2\leq b\leq n-2$. Equivalently, we have $b\geq 2$ and $n-b\geq 2$. Then it is straightforward that the pair $(0,b)$ is realized by the graph $S_{b}\sqcup K_{n-b}$ by Lemmas~\ref{lem:union-graphs}, \ref{lem:complete-graph}, and \ref{lem:star-graph}. This concludes the~proof.
\end{proof}


\section{From spectral graph theory: multiplicity of line graph of forests}\label{sec:spectral}

In this section, we examine the relationship between the independence polynomial of the line graph of a forest and two other polynomials associated to graphs, namely, characteristic polynomial and matching polynomial. We use these connections to investigate when $-1$ is a root of the independence polynomial, beginning with some preliminary definitions.

Let $G$ be a simple graph on the vertex set $\{1, \dots, n\}$. The \textit{adjacency matrix} of $G$ is the $n \times n$ symmetric matrix $A(G) = [a_{ij}]$ where
\[
a_{ij} = 
\begin{cases} 
1 & \text{if vertices } i \text{ and } j \text{ are adjacent,} \\
0 & \text{otherwise.}
\end{cases}
\]
By definition, $a_{ii} = 0$ for all $i \in \{1, \dots, n\}$, and $a_{ij} = a_{ji}$ for all $1 \leq i, j \leq n$. The {\textit{characteristic polynomial}} of $G$ is the polynomial $\phi_G(x)=\det\,(xI-A(G))$, and its roots are the {\textit{eigenvalues}} of $G$.

A {\textit{matching}} $M$ in a graph $G$ is a set of edges such that no two share a vertex in common. If $|M|=k$, then $M$ is called a $k$-{\textit{matching}}. Let $m_k(G)$ denote the number of $k$-matchings in $G$, with the convention that $m_0(G)=1$. The {\textit{matching number}} of $G$, denoted by $\nu_G$, is the maximum $k$ such that $m_k(G)\neq 0$. The {\textit{matching polynomial}} is defined by $$\mu_G(x)=\sum_{k=0}^{\nu_G}(-1)^k\, m_k(G)\, x^{n-2k}.$$

The \textit{line graph} $L(G)$ of a graph $G$ is the graph whose vertices are the edges of $G$, and where two vertices of $L(G)$ are adjacent if and only if the corresponding edges in $G$ share a common vertex. We denote by $e_{i,j}$ the vertex in $L(G)$ corresponding to the edge between vertices $i$ and $j$ in $G$. It is easy to see that $L(P_n)=P_{n-1}$ and $L(C_n)=C_n$. 

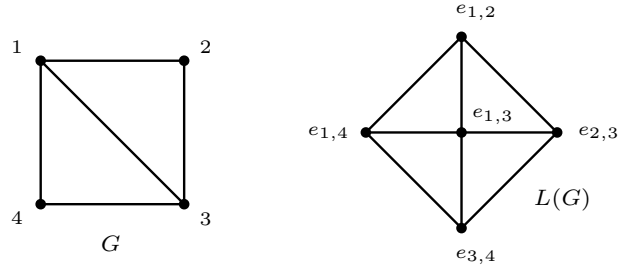
\begin{figure}[h]
    \centering
\scalebox{1.2}{

\tikzset{every picture/.style={line width=0.75pt}} 

\begin{tikzpicture}[x=0.75pt,y=0.75pt,yscale=-1,xscale=1]

\draw    (70,70) -- (70,130) ;
\draw    (130,70) -- (130,130) ;
\draw    (70,70) -- (130,70) ;
\draw    (70,130) -- (130,130) ;
\draw    (70,70) -- (130,130) ;
\draw    (246,60) -- (246,100) ;
\draw    (246,100) -- (246,140) ;
\draw    (246,60) -- (286,100) ;
\draw    (286,100) -- (246,140) ;
\draw    (206,100) -- (246,140) ;
\draw    (246,60) -- (206,100) ;
\draw    (246,100) -- (206,100) ;
\draw    (286,100) -- (246,100) ;
\draw  [fill={rgb, 255:red, 0; green, 0; blue, 0 }  ,fill opacity=1 ] (68.29,130) .. controls (68.29,129.05) and (69.05,128.29) .. (70,128.29) .. controls (70.95,128.29) and (71.71,129.05) .. (71.71,130) .. controls (71.71,130.95) and (70.95,131.71) .. (70,131.71) .. controls (69.05,131.71) and (68.29,130.95) .. (68.29,130) -- cycle ;
\draw  [fill={rgb, 255:red, 0; green, 0; blue, 0 }  ,fill opacity=1 ] (128.29,130) .. controls (128.29,129.05) and (129.05,128.29) .. (130,128.29) .. controls (130.95,128.29) and (131.71,129.05) .. (131.71,130) .. controls (131.71,130.95) and (130.95,131.71) .. (130,131.71) .. controls (129.05,131.71) and (128.29,130.95) .. (128.29,130) -- cycle ;
\draw  [fill={rgb, 255:red, 0; green, 0; blue, 0 }  ,fill opacity=1 ] (128.29,70) .. controls (128.29,69.05) and (129.05,68.29) .. (130,68.29) .. controls (130.95,68.29) and (131.71,69.05) .. (131.71,70) .. controls (131.71,70.95) and (130.95,71.71) .. (130,71.71) .. controls (129.05,71.71) and (128.29,70.95) .. (128.29,70) -- cycle ;
\draw  [fill={rgb, 255:red, 0; green, 0; blue, 0 }  ,fill opacity=1 ] (68.29,70) .. controls (68.29,69.05) and (69.05,68.29) .. (70,68.29) .. controls (70.95,68.29) and (71.71,69.05) .. (71.71,70) .. controls (71.71,70.95) and (70.95,71.71) .. (70,71.71) .. controls (69.05,71.71) and (68.29,70.95) .. (68.29,70) -- cycle ;
\draw  [fill={rgb, 255:red, 0; green, 0; blue, 0 }  ,fill opacity=1 ] (204.29,100) .. controls (204.29,99.05) and (205.05,98.29) .. (206,98.29) .. controls (206.95,98.29) and (207.71,99.05) .. (207.71,100) .. controls (207.71,100.95) and (206.95,101.71) .. (206,101.71) .. controls (205.05,101.71) and (204.29,100.95) .. (204.29,100) -- cycle ;
\draw  [fill={rgb, 255:red, 0; green, 0; blue, 0 }  ,fill opacity=1 ] (244.29,140) .. controls (244.29,139.05) and (245.05,138.29) .. (246,138.29) .. controls (246.95,138.29) and (247.71,139.05) .. (247.71,140) .. controls (247.71,140.95) and (246.95,141.71) .. (246,141.71) .. controls (245.05,141.71) and (244.29,140.95) .. (244.29,140) -- cycle ;
\draw  [fill={rgb, 255:red, 0; green, 0; blue, 0 }  ,fill opacity=1 ] (284.29,100) .. controls (284.29,99.05) and (285.05,98.29) .. (286,98.29) .. controls (286.95,98.29) and (287.71,99.05) .. (287.71,100) .. controls (287.71,100.95) and (286.95,101.71) .. (286,101.71) .. controls (285.05,101.71) and (284.29,100.95) .. (284.29,100) -- cycle ;
\draw  [fill={rgb, 255:red, 0; green, 0; blue, 0 }  ,fill opacity=1 ] (244.29,60) .. controls (244.29,59.05) and (245.05,58.29) .. (246,58.29) .. controls (246.95,58.29) and (247.71,59.05) .. (247.71,60) .. controls (247.71,60.95) and (246.95,61.71) .. (246,61.71) .. controls (245.05,61.71) and (244.29,60.95) .. (244.29,60) -- cycle ;
\draw  [fill={rgb, 255:red, 0; green, 0; blue, 0 }  ,fill opacity=1 ] (244.29,100) .. controls (244.29,99.05) and (245.05,98.29) .. (246,98.29) .. controls (246.95,98.29) and (247.71,99.05) .. (247.71,100) .. controls (247.71,100.95) and (246.95,101.71) .. (246,101.71) .. controls (245.05,101.71) and (244.29,100.95) .. (244.29,100) -- cycle ;

\draw (56.15,60.18) node [anchor=north west][inner sep=0.75pt]  [font=\tiny]  {$1$};
\draw (135.04,60.2) node [anchor=north west][inner sep=0.75pt]  [font=\tiny]  {$2$};
\draw (135,132.2) node [anchor=north west][inner sep=0.75pt]  [font=\tiny]  {$3$};
\draw (56.2,131.51) node [anchor=north west][inner sep=0.75pt]  [font=\tiny]  {$4$};
\draw (242,45.29) node [anchor=north west][inner sep=0.75pt]  [font=\tiny]  {$e_{1,2}$};
\draw (248.89,88.62) node [anchor=north west][inner sep=0.75pt]  [font=\tiny]  {$e_{1,3}$};
\draw (293.22,96) node [anchor=north west][inner sep=0.75pt]  [font=\tiny]  {$e_{2,3}$};
\draw (180.22,95.96) node [anchor=north west][inner sep=0.75pt]  [font=\tiny]  {$e_{1,4}$};
\draw (242,147.4) node [anchor=north west][inner sep=0.75pt]  [font=\tiny]  {$e_{3,4}$};
\draw (93.6,141.6) node [anchor=north west][inner sep=0.75pt]  [font=\scriptsize]  {$G$};
\draw (274.8,122) node [anchor=north west][inner sep=0.75pt]  [font=\scriptsize]  {$L( G)$};

\end{tikzpicture}

}    
    \caption{The diamond graph $G$ and its line graph $L(G)$}
    \label{fig:glg}
\end{figure}

Line graphs of trees have a particularly simple structure. The following lemma provides a characterization of line graphs of trees in terms of block graphs.

\begin{lemma}[\protect{\cite{fdb8a226-e07b-3d43-b47f-a71e676a165a}}]\label{rem:line-tree}
A graph is the line graph of a tree if and only if it is a connected block graph in which every cutpoint belongs to exactly two blocks. 
\end{lemma}

The following classical result establishes a bridge between matchings in $G$ and independent sets in $L(G)$. For the sake of completion, we provide a self-contained proof.

\begin{theorem}
    For any graph $G$,
    $$P_{L(G)}(x)= \sum_{k\geq 0} m_k(G)\, x^k$$
\end{theorem}

\begin{proof}
By definition, an independent set of size $k$ in $L(G)$ is a set of $k$ vertices of $L(G)$ with no two of which are adjacent. Since the vertices of $L(G)$ correspond to the edges of $G$, two vertices in $L(G)$ are adjacent if and only if the corresponding edges in $G$ share an endpoint. Therefore, an independent set of size $k$ in $L(G)$ corresponds precisely to a matching of size $k$ in $G$. This completes the proof.
\end{proof}

\begin{remark}\label{rem:m=i}
 Let $Q_G(x):=\sum_{k\geq 0} m_k(G)\, x^k$. Then it is easy to see that
    $\mu_G(x) = x^n\, Q_G(-x^{-2}).$ Thus, \begin{equation}
   \mu_G(x)=x^n\, P_{L(G)} (-x^{-2}). 
\end{equation}  
\end{remark}

The matching polynomial of a graph encodes information about its matchings, while the characteristic polynomial of a graph captures its spectral properties. For general graphs these polynomials differ, but they coincide precisely when the graph has no cycles.

\begin{theorem}[\protect{\cite{GodsilGutman1981}}]\label{m=a} For a graph $G$,
 $\phi_G(x)=\mu_G(x)$ if and only if $G$ is a forest.
\end{theorem}

By Remark \ref{rem:m=i} and Theorem \ref{m=a}, it follows that for any forest $F$ on $n$ vertices, 
\begin{equation}\label{eq:char=in}
    \phi_F(x) = x^n P_{L(F)} (-x^{-2}).
\end{equation}
Putting $x=1$ and $x=-1$ in \eqref{eq:char=in}, we see that if either $1$ or $-1$ is an eigenvalue of $F$, then $P_{L(F)}(-1)=0$. Hence, $-1$ is a root of the independence polynomial of $L(F)$.
We now compare the corresponding multiplicities. Since $F$ is a forest, it is bipartite, and therefore its nonzero eigenvalues occur in pairs $\pm\lambda$ with the same multiplicity. In particular, the multiplicities of $1$ and $-1$ as eigenvalues of $F$ are equal. Moreover, the factor $x^n$ in \eqref{eq:char=in} is nonzero at $x=1$ and at $x=-1$. The change of variable $t=-x^{-2}$ is locally invertible at both points, since $\frac{d}{dx}(-x^{-2})=2x^{-3},$ is nonzero at $x=1$ and at $x=-1$. Therefore, the order of vanishing of $P_{L(F)}(t)$ at $t=-1$ is equal to the order of vanishing of $\phi_F(x)$ at $x=1$, and also to the order of vanishing of $\phi_F(x)$ at $x=-1$. Consequently, the multiplicity of $-1$ as a root of $P_{L(F)}$ is equal to the common multiplicity of $1$ and $-1$ as eigenvalues of $F$.


Let the connected components of $F$ be the trees $T_1, \dots, T_r$. Since the characteristic polynomial of a disjoint union is the product of the characteristic polynomials of its components, we have $\phi_F(x) = \prod_{i=1}^r \phi_{T_i}(x)$. Thus, $\pm1$ is a root of $\phi_{F}(x)$ if and only if it is a root of $\phi_{T_i}(x)$ for some $1 \leq i \leq r$. 

Similarly, the connected components of $L(F)$ are exactly the line graphs of the components of $F$, namely $L(T_1), \dots, L(T_r)$. It then follows that $P_{L(F)}(x) = \prod_{i=1}^{r} P_{L(T_i)}(x)$. Consequently, $\pm1$ is a root of $P_{L(F)}(x)$ if and only if $\pm1$ is a root of $P_{L(T_i)}(x)$ for some $1 \leq i \leq r$.

Our goal is to study the case when $-1$ is a root of the independence polynomial of the line graph of a forest $F$. By the preceding discussion, it suffices to consider the case where $F$ is a tree and $-1$ is its eigenvalue.

For trees, two characterizations of the multiplicity of $-1$ as an eigenvalue are known in terms of the number of pendant vertices. Using these characterizations, we determine line graphs whose independence polynomials have $-1$ as a root of a given multiplicity. Before proceeding, we introduce some terminology.

Let $T$ be a tree on $n$ vertices with $p$ pendant vertices. A {\textit{major vertex}} in $T$ is a vertex that is adjacent to at least three other vertices. For any two vertices $x$ and $y$ in $T$, the {\textit{distance}} $d(x,y)$ is defined as the length of the unique path between them. 

In \cite{MR4011062}, Wang et. al.\ established an upper bound for $\mult_{\lambda} \phi_T$ in terms of the number of pendant vertices as follows.

\begin{theorem}  [\protect{\cite[Corollary~2.10]{MR4011062}}]
Let $T$ be a tree with $p$ pendant vertices. Then \[\mult_{\lambda} \phi_T \leq p - 1.\] 
\end{theorem}

In \cite{MR4506593}, the authors provide a complete characterization of the trees that attain this upper bound for $\lambda=-1$.

\begin{theorem}[\protect{\cite[Theorem~2.9]{MR4506593}}]\label{thm:p-1} Let $T$ be a tree with $p \geq 2$ pendant vertices. Then $\mult_{-1} \phi_T=p-1$ if and only if one of the following conditions holds.
    \begin{enumerate}
        \item[\rm{(i)}] $T \cong P_n$ with $n \equiv 2 \pmod 3$; 
        \item[\rm{(ii)}] $d(v, u) \equiv 2 \pmod 3$ for every pendant vertex $v$ and major vertex $u$ of $T$.
    \end{enumerate}
\end{theorem}





\begin{figure}[h]
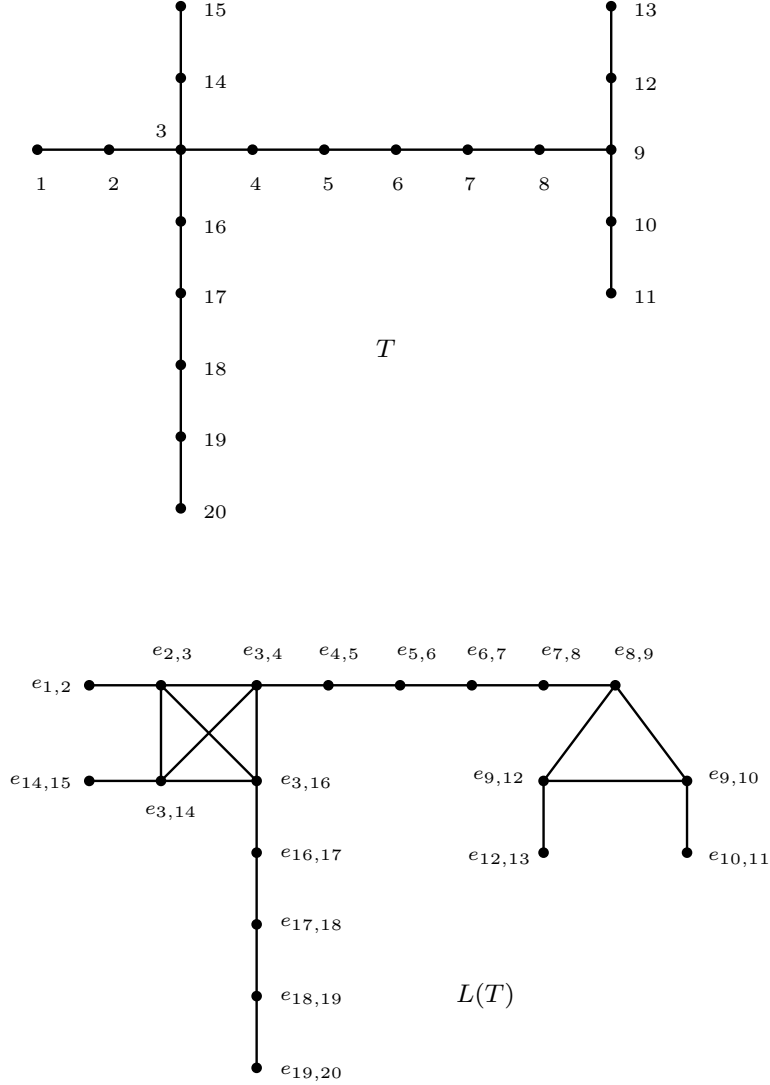

    \centering
\scalebox{1.2}{

\tikzset{every picture/.style={line width=0.75pt}} 


}
\caption{A tree $T$ and its line graph $L(T)$}
\label{fig:TandLT}
\end{figure}

We elucidate this through an example.

\begin{example}
  Consider the tree $T$ illustrated in Figure \ref{fig:TandLT}. This tree has five pendant vertices $\{1, 11, 13, 15, 20\}$ and two major vertices $\{3, 9\}$. Since $T$ satisfies condition (ii) of Theorem \ref{thm:p-1}, the multiplicity of the eigenvalue $-1$ is $\mult_{-1} \phi_T = 4$. Furthermore, Equation \eqref{eq:char=in} implies that $-1$ is also a root of the independence polynomial $P_{L(T)}(x)$ with multiplicity four.

Direct computation yields the characteristic polynomial of $T$:
\begin{equation*}
    \phi_T(x) = (x-1)^4(x+1)^4(x^{12}-15x^{10}+83x^8-204x^6+202x^4-39x^2+1),
\end{equation*}
and the independence polynomial of the line graph $L(T)$:
\begin{equation*}
    P_{L(T)}(x) = (x+1)^4(x^6+39x^5+202x^4+204x^3+83x^2+15x+1).
\end{equation*}
It is easily verified that these polynomials satisfy the relationship $\phi_T(x) = x^{20} P_{L(T)}(-x^{-2})$.
\end{example}

In \cite{chang2024characterizationtreestmt}, the authors characterize trees that have an eigenvalue $\lambda$ with multiplicity two less than the number of pendant vertices. To provide this characterization, they recursively define two families of trees, $\{\Gamma_i(\lambda)\}_{i\geq0}$ and $\{\Gamma_i^2(\lambda)\}_{i\geq0}$. We define these families for the particular case $\lambda=-1$, and for the sake of simplicity, we write $\Gamma_i(-1)$ as $\Gamma_i$ and $\Gamma_i^2(-1)$ as $\Gamma_i^2$. The first family $\{\Gamma_j\}_{j\geq 0}$ of trees is defined as follows.

\begin{itemize}
    \item $\Gamma_0:=\{P_n:n\equiv 2 \pmod 3\}.$
    
    \item $\Gamma_1$ denotes the set of trees $T$ containing a unique major vertex $w_1$ such that the forest $T-w_1$, obtained by deleting $w_1$, is the union of at least three components from $\Gamma_0$, where the neighbor of $w_1$ in each component is a pendant vertex of that component.

    
    \item For $j \ge 2$, $\Gamma_j$ consists of all trees $T$ satisfying:
    \begin{enumerate}
        \item $T$ has exactly $j$ major vertices;
        \item there exists a major vertex $w_j$ of $T$ such that $T - w_j$ has exactly one component from $\Gamma_{j-1}$ and all other components belong to $\Gamma_0$;
        \item the neighbor of $w_j$ in each component of $T - w_j$ is a pendant vertex of that component.
    \end{enumerate}
\end{itemize}


\begin{figure}[h!]
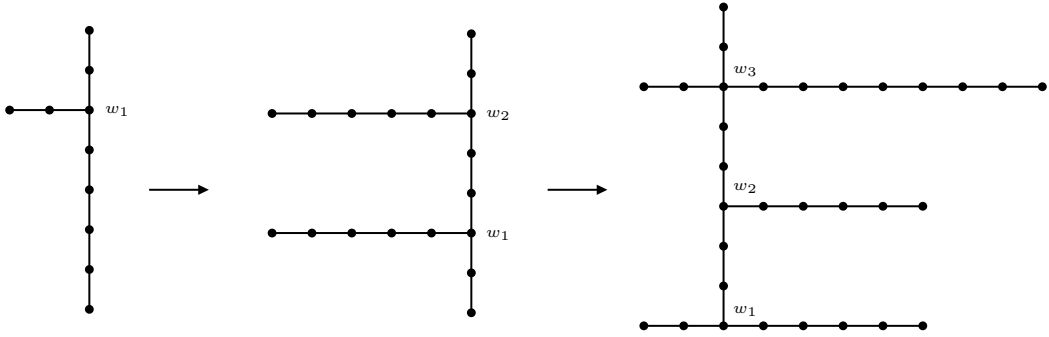

    \centering

\tikzset{every picture/.style={line width=0.75pt}} 


    
\caption{Construction of a tree in $\Gamma_3$.}
\label{fig:gamma_3}
\end{figure}

Using the family $\{\Gamma_j\}$, we now define the second family $\{\Gamma_j^2\}_{j\geq 0}$ of trees.

\begin{itemize}
    \item $\Gamma_0^2:=\{P_n:n\equiv 1 \pmod 3\}.$
    \item $\Gamma_1^2$ denotes the set of trees $T$ with a unique major vertex $w_1$ satisfying either of the following conditions:
    \begin{enumerate}
        \item $T-w_1$ has exactly three components, all from $\Gamma_0^2$ such that the neighbor of $w_1$ in each component is a pendant vertex of that component.
        \item $T-w_1$ has exactly one component from $\Gamma_0^2$ and all other components from $\Gamma_0$ such that the neighbor of $w_1$ in each component is a pendant vertex of that component.
    \end{enumerate}
 \item $\Gamma_2^2$ denotes the set of trees $T$ with exactly two major vertices, and there is a major vertex $w_2$ of $T$ such that either of the following conditions holds:
 \begin{enumerate}
     \item $T-w_2$ has exactly one component from $\Gamma_1^2$ and all the other components from $\Gamma_0$ such that the neighbor of $w_2$ in each component is a pendant vertex of that component.
     \item $T-w_2$ has exactly one component, say $T_1$, from $\Gamma_1$ and all other components from $\Gamma_0$, where the unique neighbor of $w_2$ lying in $T_1$ is not a pendant vertex of $T_1$; other neighbors of $w_2$ in each component are pendant vertices of those components.
 \end{enumerate}

 \begin{figure}[h!]
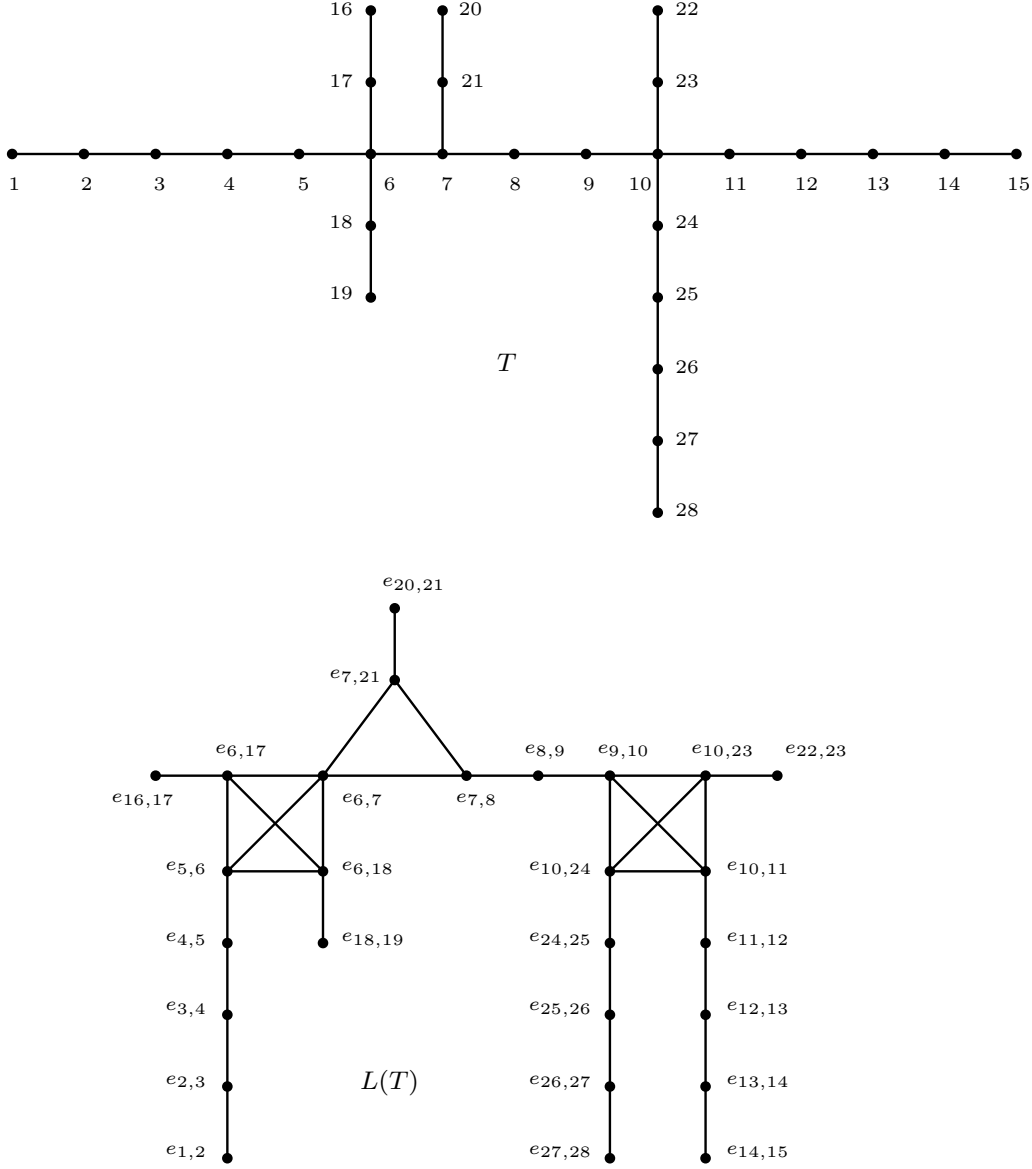

    \centering
 \scalebox{1.2}{
\tikzset{every picture/.style={line width=0.75pt}} 


}   
\caption{A tree $T\in \Gamma_3^2$ and its line graph $L(T)$}
\label{fig:gamma_32}
\end{figure}   
 \item For $j \ge 3$, $\Gamma_j^2$ consists of all trees $T$ with exactly $j$ major vertices, and there is a major vertex $w_j$ such that one of the following conditions holds: 
    \begin{enumerate}
    \item $T-w_j$ has exactly one component from $\Gamma_{j-1}^2$ and all other components from $\Gamma_0$ such that the neighbor of $w_j$ in each component is a pendant vertex of that component.
    \item $T-w_j$ has exactly one component, say $T_1$, from $\Gamma_{j-1}$ and all other components from $\Gamma_0$, where the unique neighbor of $w_j$ lying in $T_1$ is not a pendant vertex of $T_1$; other neighbors of $w_j$ in each component are pendant vertices of those components.
    \item $T-w_j$ has exactly one component, say $T_2$, from $\Gamma_{j-1}$; exactly one component, say $T_2$ from $\Gamma_{0}^2$ and all other components from $\Gamma_0$  such that the neighbor of $w_j$ in each component is a pendant vertex of that component.
    \end{enumerate}
\end{itemize}

The following result from \cite{chang2024characterizationtreestmt} provides a characterization of trees that have $-1$ as an eigenvalue with a multiplicity exactly two less than the number of pendant vertices.

\begin{theorem}
[\protect{\cite[Theorems~1.4 and 1.5]{chang2024characterizationtreestmt}}]\label{thm:p-2} Let $T$ be a tree with $p \geq 3$ pendant vertices and $m$ major vertices, and suppose that $-1$ is an eigenvalue of $T$. Then
$\mult_{-1} \phi_T=p-2
$
if and only if $T \in \Gamma_k^2$.
\end{theorem}


\begin{example}
Consider the tree $T$ in Figure~\ref{fig:gamma_32}. Observe that $T \in \Gamma_3^2$ with $w_3=10$, and that $T$ has seven pendant vertices. By Theorem~\ref{thm:p-2}, we obtain $\mult_{-1} \phi_T=5.$ By Equation~\eqref{eq:char=in}, it follows that $\mult_{-1} P_{L(T)}=5$. 

Direct computation yields the characteristic polynomial of $T$: $$\phi_T(x)=x^2 (x - 1)^5 (x + 1)^5 (x^2 - 3)^2 (x^{12} - 16x^{10} + 93x^8 - 237x^6 + 249x^4 - 70x^2 + 5)$$
 and the independence polynomial of the line graph $L(T)$:
  $$P_{L(T)}(x)=(3x + 1)^2  (x + 1)^5  (5x^6 + 70x^5 + 249x^4 + 237x^3 + 93x^2 + 16x + 1).$$
It is easily verified that these polynomials satisfy the relationship $\phi_T(x) = x^{28} P_{L(T)}(-x^{-2})$.  
\end{example}


\section{To commutative algebra: $\mathfrak{a}$-invariant and regularity}\label{sec:algebra}

Let $\mathbb{K}$ be a field, and $A = \bigoplus_{i\geq 0}A_i$
be a standard graded $\mathbb{K}$-algebra. The Hilbert series of $A$ is defined as
\[H_A(t) = \sum_{i \geq 0} \mathrm{dim}_{\mathbb{K}}A_i~t^i.\]
 Let $d$ be the krull dimension of $A$, it is known that the Hilbert series of $A$ can be written as a rational function 
\[H_A(t) = \frac{h_A(t)}{(1-t)^d},\]
for a unique integer polynomial $h_A(t)$ with $h_A(1)\neq 0$. The polynomial $h_A(t)$ is known as the \emph{$h$-polynomial} of $A$, and the difference between the degree of $h$-polynomial and the dimension of $A$ (i.e. $\mathrm{deg}~h_A(t) - \mathrm{dim}~A$) is known to be the \emph{$\mathfrak{a}$-invariant} of $A$, denoted by $\mathfrak{a}(A)$. 

Throughout this section let $R = \mathbb{K}[x_1,\ldots,x_n]$ denote the standard graded polynomial ring in $n$ variables over the field $\mathbb{K}$. Let $G$ be a finite simple graph, then one can associate a quadratic square-free monomial ideal $I(G)$ to $G$, known as the \emph{edge ideal} of $G$, as following:
\[I(G) = (x_ix_j \mid \{i,j\} \in E(G)) \subseteq R.\] 

By \cite[Corollary B.4.1]{Vasconcelos98}, one has
\begin{equation} \label{h-degree_reg_ineq}
\mathrm{deg}~h_{R/I(G)}(t) - \operatorname{reg}(R/I(G)) \leq  \mathrm{dim}(R/I(G)) -  \mathrm{depth}(R/I(G)).
\end{equation}

Since,  $\mathrm{deg}~h_{R/I(G)}(t)= \mathfrak{a}(R/I(G)) + \mathrm{dim}(R/I(G))$, we have  
$$
\operatorname{reg}(R/I(G)) \geq \mathfrak{a}(R/I(G))  +  \mathrm{depth}(R/I(G)).
$$

Furthermore, if $I(G)$ is a Cohen--Macaulay ideal \cite[Lemma 3]{Bigedi_Herzog17} or $I(G)$ has a pure resolution \cite[Pg 153]{Bruns_Herzog93}, then equality holds in (\ref{h-degree_reg_ineq}). In particular, if $I(G)$ is a Cohen--Macaulay ideal, then 
$$
\operatorname{reg}(R/I(G)) = \mathfrak{a}(R/I(G)) + \mathrm{dim}(R/I(G))
$$

For a graph $G$, it is known that the independence number $\alpha(G)$ is equal to the Krull dimension of $R/I(G)$~(e.g., see \cite[Remark 2.11]{biermann-et-al.2026}). Also, by \cite[Theorem 4.4]{biermann-et-al.2026}, we have
\[ \mathfrak{a}(R/I(G)) + \mathrm{dim}(R/I(G)) = \mathrm{deg}~h_{R/I(G)}(t) = \alpha(G)- \mult_{-1}P_G.\]
Thus, one can observe that $\mathfrak{a}(R/I(G)) = - \mult_{-1}P_G.$ Using this correspondence, our results from previous sections can be translated directly into interesting algebraic results. 

Our first algebraic result translates Theorem~\ref{thm:branch-algorithm}, concerning the vanishing of the $\mathfrak{a}$-invariant of the graded algebra $R/I(G)$ associated with pseudo-forests.

\begin{theorem}
Let $G$ be a pseudo-forest and $v_1,\dots, v_t$ a maximal support sequence of $G$ for some integer $t$. Let $H = G \setminus \cup_{i=1}^t N_G[v_i]$. Then the $\mathfrak{a}$-invariant of  $R/I(G)$ vanishes if and only if $H$ has no isolated vertex. 
\end{theorem}


The next result translates Theorem~\ref{thm:tree-description}, which provides a structural characterization of trees for which the $\mathfrak{a}$-invariant vanishes.

\begin{theorem}
    Let $T$ denote a tree. The $\mathfrak{a}$-invariant of $R/I(T)$ is non-zero if and only if $T$ is isomorphic to a grafting on
    some other tree $T'$.
\end{theorem}

As a consequence of the previous results, we obtain families of graphs for which the Castelnuovo--Mumford regularity of $R/I(G)$ is bounded below by the depth of $R/I(G)$. Note that, in general, there is no direct relationship between $\operatorname{reg}(R/I(G))$ and $\operatorname{depth}(R/I(G))$, even when $G$ is a tree. In fact, for trees, the depth can exceed the regularity by an arbitrarily large amount.

To see this, let $P_n$ denote the path on $n$ vertices and consider the whiskered graph $G=\mathcal{W}(P_n)$. By Lemma~\ref{lem:G*}, we have
\[
\dim(R/I(G))=\alpha(G)=n.
\]
Since the whiskering of any graph is Cohen--Macaulay (see \cite{Villarreal}), it follows that
\[
\operatorname{depth}(R/I(G))=\dim(R/I(G))=n.
\]
On the other hand, by \cite[Lemma~21]{Woodroofe},
\[
\operatorname{reg}(R/I(G) = \left\lfloor \frac{n+1}{2}\right\rfloor.
\]
Therefore,
\[
\operatorname{depth}(R/I(G)) - \operatorname{reg}(R/I(G)) = n-\left\lfloor \frac{n+1}{2}\right\rfloor,
\]
which grows arbitrarily large as $n$ increases. Consequently, even within the class of trees, the depth of $R/I(G)$ can be substantially larger than its Castelnuovo--Mumford regularity.

The following result provides nontrivial families of graphs for which the inequality goes in the opposite direction, namely,
$\operatorname{reg}(R/I(G)) \geq \operatorname{depth}(R/I(G)).$


\begin{corollary}
Let $G$ be a graph. Then
$$
\operatorname{reg}(R/I(G)) \geq \mathrm{depth}(R/I(G))
$$
in each of the following cases:
\begin{enumerate}
    \item $G$ is a pseudo-forest admitting a maximal support sequence $v_1,\dots,v_t$ such that the graph
    $ G \setminus \cup_{i=1}^t N_G[v_i]$ has no isolated vertex;
    
    \item $G$ is a tree that is not a grafting on another tree.
\end{enumerate}
\end{corollary}

The following result provides lower bounds for the $\mathfrak{a}$-invariant and the Castelnuovo--Mumford regularity of connected graphs whose edge ideals satisfy $\dim(R/I(G)) = n-2$. These bounds are obtained as a consequence of Theorem \ref{thm:alpha=n-2} and depend only on the number of vertices of $G$.

\begin{theorem}
Let $G$ be a connected graph on $n \geq 2$ vertices such that $\dim(R/I(G)) = n-2$. Then
$$
\mathfrak{a}(G) \geq 1-\left\lfloor \frac{n}{2}\right\rfloor,
$$
and if $I(G)$ is Cohen--Macaulay or admits a pure resolution,
$$
\operatorname{reg}(R/I(G)) \geq \left\lceil \frac{n}{2}\right\rceil -1.
$$
\end{theorem}

The following result is immediate from the correspondence between the $\mathfrak{a}$-invariant and multiplicity of graph, dimension and the independence number, and the definitions of the sets $\mathcal{MI}(n)$ and $\mathcal{MI}^c(n)$.

\begin{theorem}
     A lattice point $p$ in $\mathbb{Z}^2$ can be realized as a pair $(\mathfrak{a}(R/I(G)), \mathrm{dim}~R/I(G))$ for some graph $G$ on $n$ vertices if and only if $p$ belongs to the reflection of the set $\mathcal{MI}(n)$ along ordinate axis, and can be realized as a pair $(\mathfrak{a}(R/I(G)), \mathrm{dim}~R/I(G))$ for some connected graph $G$ on $n$ vertices if and only if $p$ belongs to the reflection of the set $\mathcal{MI}^c(n)$ along ordinate axis.
\end{theorem}

Now, from our study of the sets $\mathcal{MI}(n)$ and $\mathcal{MI}^c(n)$, we can classify the lattice points in $\mathbb{Z}^2$ that can be realized as pairs consisting of the $\mathfrak{a}$-invariant and the dimension of the algebra $R/I(G)$ for a graph $G$ on $n$ vertices. The following result follows from the Theorem~\ref{thm:MI-n-all-graphs}.

\begin{theorem}
   Let $n$ be a fixed positive integer. Let   $p$ be a lattice point in $\mathbb{Z}^2$. Then $p = (\mathfrak{a}(R/I(G)), \mathrm{dim}~R/I(G))$ for some graph $G$ on $n$ vertices if and only if $$p \in \{(a,b) \in Q_2 \cap \mathbb{Z}^2 \mid 0 \leq -a < b \leq n-1\} \cup \{(-n,n)\},$$
   where $Q_2$ denotes the upper-left quadrant of the Euclidean plane.
\end{theorem}

As we do not know the set $\mathcal{MI}^c(n)$ completely, we can not completely classify the lattice points in $\mathbb{Z}^2$ that can be realized as pairs consisting of the $\mathfrak{a}$-invariant and the dimension of the algebra $R/I(G)$ for a connected graph $G$ on $n$ vertices. However, when we do not have any restriction on number of vertices, we can completely classify the lattice points in $\mathbb{Z}^2$ that can be realized as pairs consisting of the $\mathfrak{a}$-invariant and the dimension of the algebra $R/I(G)$ for some connected graph $G$. The following result follows from the Theorems~\ref{cor:MI-all-connected} and \ref{thm:MI-n-all-graphs}.

\begin{theorem}
  Let $p$ be a lattice point in $\mathbb{Z}^2$. Then
  \begin{enumerate}
      \item $p = (\mathfrak{a}(R/I(G)), \mathrm{dim}~R/I(G))$ for some graph $G$ if and only if $$p \in \{(a,b) \in Q_2 \cap \mathbb{Z}^2 \mid 0 \leq -a \leq b\}\setminus \{(0,0)\}.$$

      \item $p = (\mathfrak{a}(R/I(G)), \mathrm{dim}~R/I(G))$ for some connected graph $G$ if and only if $$p \in \{(a,b) \in Q_2 \cap \mathbb{Z}^2 \mid 0 \leq -a < b\} \cup \{(-1,1)\}.$$
   \end{enumerate}
\end{theorem}

\bibliographystyle{amsplain}
\bibliography{refs}

\end{document}